\documentclass{amsart}

\usepackage{amssymb}
\usepackage[utf8]{inputenc}
\usepackage{tikz}
\usepackage{tikz-cd}
\usepackage{amsmath, amsthm, amssymb, amsfonts}
\usepackage{hyperref}
\usepackage{mathtools}
\usepackage[retainorgcmds]{IEEEtrantools}
\usepackage{latexsym}
\usepackage{enumerate}
\usepackage{nomencl}
\usepackage{nameref}
\usepackage{chngcntr}

\makeatletter
\let\orgdescriptionlabel\descriptionlabel
\renewcommand*{\descriptionlabel}[1]{%
  \let\orglabel\label
  \let\label\@gobble
  \phantomsection
  \edef\@currentlabel{#1}%
  \let\label\orglabel
  \orgdescriptionlabel{#1}%
}
\makeatother

\newtheorem{thm}{Theorem}[section]
\newtheorem{lem}[thm]{Lemma}
\newtheorem{prop}[thm]{Proposition}
\newtheorem{cor}[thm]{Corollary}
\newtheorem{innercustomthm}{Theorem}
\newenvironment{customthm}[1]
  {\renewcommand\theinnercustomthm{#1}\innercustomthm}
  {\endinnercustomthm}

\theoremstyle{definition}
\newtheorem{defi}[thm]{Definition}

\theoremstyle{remark}
\newtheorem{rem}[thm]{Remark}

\newcommand{\op}{\normalfont{^{op}}}

\newcommand{\Mod}[1]{\operatorname{mod} #1}

\newcommand{\Hom}[3]{\operatorname{Hom}_{#1}\left(#2,#3\right) }
\newcommand{\Tra}[2]{\operatorname{Tr}(#1,#2)}

\newcommand{\Tom}[3]{\operatorname{Hom}_{#1}(#2,#3) }
\newcommand{\End}[2]{\operatorname{End}_{#1}\left(#2\right) }
\newcommand{\END}[2]{\operatorname{End}_{#1}(#2) }

\newcommand{\Soc}[1]{\operatorname{Soc}#1}
\newcommand{\Top}[1]{\operatorname{Top}#1}

\newcommand{\Rad}[1]{\operatorname{Rad} #1}
\newcommand{\RAD}[2]{\operatorname{Rad}^{#1} #2}
\newcommand{\SOC}[2]{\operatorname{Soc}_{#1} #2}
\newcommand{\USOC}[2]{\underline{\operatorname{Soc}}_{#1} #2}
\newcommand{\Ima}[1]{\operatorname{Im}#1}
\newcommand{\Ker}[1]{\operatorname{Ker}#1}

\newcommand{\Ext}[4]{\operatorname{Ext}_{#1}^{#2}\left(#3,#4 \right)}
\newcommand{\EXT}[4]{\operatorname{Ext}_{#1}^{#2}(#3,#4 )}

\newcommand{\Z}{\mathbb{Z}}

\newcommand{\Zp}{\Z_{> 0}}

\newcommand{\B}[1]{\left(#1\right)}

\newcommand{\Gl}[1]{\operatorname{gl.dim}#1}
\newcommand{\dssl}[1]{\operatorname{\Delta.ssl}#1}

\newcommand{\Add}[1]{\operatorname{add}#1}

\newcommand{\LL}[1]{\operatorname{LL}(#1)}

\begin{document}

\title[The Ringel dual of the ADR algebra]{The Ringel dual of the Auslander-Dlab-Ringel algebra}
\author{Teresa Conde}
\address{Institute of Algebra and Number Theory, University of Stuttgart\\ Pfaffenwaldring 57, 70569 Stuttgart, Germany}
\email{\href{mailto:tconde@mathematik.uni-stuttgart.de}{\nolinkurl{tconde@mathematik.uni-stuttgart.de}}}
\author{Karin Erdmann}
\address{Mathematical Institute,
University of Oxford\\ Radcliffe Observatory Quarter,
OX2 6GG, United Kingdom}
\email{\href{mailto:erdmann@maths.ox.ac.uk}{\nolinkurl{erdmann@maths.ox.ac.uk}}}
\thanks{The first named author was supported by the grant SFRH/BD/84060/2012 of Funda\c{c}\~ao para a Ci\^encia e a Tecnologia while part of this work was carried out.}
\subjclass[2010]{Primary 16S50, 16W70. Secondary 16G10, 16G20.}
\keywords{Quasihereditary algebra, Ringel dual, ADR algebra}
\date{\today}

\begin{abstract}
The ADR algebra $R_A$ of a finite-dimensional algebra $A$ is a quasihereditary algebra. In this paper we study the Ringel dual $\mathcal{R}(R_A)$ of $R_A$. We prove that $\mathcal{R}(R_A)$ can be identified with $\B{R_{A\op}}\op$, under certain `minimal' regularity conditions for $A$. We also give necessary and sufficient conditions for the ADR algebra to be Ringel selfdual.
\end{abstract}

\maketitle
\tableofcontents

\section{Introduction}
\label{sec:overviewpart2}

Quasihereditary algebras were introduced by Cline, Parshall and Scott (\cite{clineparshallscott}) to investigate highest weight categories arising in algebraic Lie theory, and they were extensively studied from the perspective of finite-dimensional algebras by Dlab and Ringel (\cite{MR1211481,MR987824}). Since then, quasihereditary algebras have been discovered in many other contexts. 

In particular, to every finite-dimensional algebra $A$ (or more generally, to every Artin algebra $A$) there is a canonical quasihereditary algebra $\tilde{R}_A$ which contains $A$ as an idempotent subalgebra. That is, there is an idempotent $\xi$ in $\tilde{R}_A$ such that  $A = \xi \tilde{R}_A \xi$. The algebra $\tilde{R}_A$ was introduced by Auslander in \cite{MR0349747}, to show that any algebra occurs as idempotent subalgebra of an algebra of finite global dimension. Subsequently, Dlab and Ringel showed in \cite{MR943793} that $\tilde{R}_A$ is in fact quasihereditary. To define $\tilde{R}_A$, take the direct sum of all radical powers of $A$, that is
\[\tilde{G}=\bigoplus_{i= 1}^{L} A/\RAD{i}{A},
\] 
where $L$ denotes the Loewy length of $A$. Then $\tilde{R}_A:=\END{A}{\tilde{G}}\op$. To study its representation theory, one considers the basic version of $\tilde{R}_A$ instead. We denote such basic algebra by $R_A$ and call it the Auslander--Dlab--Ringel algebra (ADR algebra) associated to $A$.

In specific cases, an algebra $A$ may occur in many different ways as an idempotent subalgebra of some quasihereditary algebras. A famous example is the Schur algebra $S\B{n,r}$. For $n\geq r$, it has the group algebra of a symmetric group of degree $r$ as an idempotent subalgebra. One might think of ADR algebras as analogues of Schur algebras, which exist in general.

\bigskip

In \cite{MR1211481}, Ringel proved that quasihereditary algebras come in pairs. For every quasihereditary algebra $\B{B, \Phi, \sqsubseteq}$, there is another quasihereditary algebra $\B{\mathcal{R}\B{B}, \Phi, \sqsubseteq\op}$, unique up to isomorphism, such that $\mathcal{R}\B{\mathcal{R}\B{B}}$ is again Morita equivalent to $B$. The algebra $\mathcal{R}\B{B}$ is defined as the endomorphism algebra $\End{B}{T}\op$ of a special $B$-module $T$, called the characteristic tilting module. We call $\mathcal{R}\B{B}$ a Ringel dual of $B$. 


The aim of this paper is to study the Ringel dual $\mathcal{R}\B{R_A}$ of an ADR algebra $R_A$. Quasihereditary algebras arising in Lie theory are often isomorphic to their own Ringel dual. For example, this is the case for category $\mathcal{O}$ (\cite{MR1445716}), or for Schur algebras $S(n,r)$ for $n\geq r$, and $q$-analogues (\cite{MR1707336}). However, Ringel selfduality appears to be less prevalent in general.

Nevertheless, there is a natural description of the Ringel dual $\mathcal{R}\B{R_A}$ of an ADR algebra $R_A$ when the radical structure of $A$ satisfies certain symmetry conditions. Our main result is as follows.

\begin{customthm}{A}[Main Theorem]
\label{customthm:A}
Let $A$ be a finite-dimensional algebra with Loewy length $L$ and assume that all projective and injective indecomposable $A$-modules are rigid with Loewy length $L$. The Ringel dual of the quasihereditary algebra $R_A$ is isomorphic to the opposite of the ADR algebra of $A\op$. That is,
\[\mathcal{R}\B{R_A}  \cong  \B{R_{A\op}}\op.
\]
\end{customthm} 

Here a module is said to be rigid if its radical series coincides with its socle series. Connected selfinjective algebras with radical cube zero but radical square nonzero trivially satisfy the conditions of Theorem \ref{customthm:A}. Such class of algebras was studied in \cite{ES}, and contains blocks of symmetric group algebras of weight $1$ (see \cite[Section $7$]{MR3510398} for an example). According to \cite{Scopes2}, blocks of symmetric group algebras of weight $2$ also satisfy the assumptions of Theorem \ref{customthm:A} when the underlying field has characteristic $p>2$.

Along the way, we prove preliminary results of independent interest. For instance, in Theorem \ref{thm:claim2}, a complete description of the $\Delta$-filtrations of the tilting $R_A$-modules is given in terms of the socle series of the injective indecomposable $A$-modules, when $A$ is an algebra whose projectives are rigid modules with Loewy length $L$ (here $L$ denotes the Loewy length of $A$). Philosophically speaking, the technicalities encountered towards the proof of Theorem \ref{customthm:A} are due to the fact that we are seeking to identify two algebras in a ``noncanonical way".

Furthermore, we give some details to explain why the assumptions on the Loewy length and on the rigidity are needed for Theorem \ref{customthm:A} to hold. First, it is necessary that the projective cover $P_i$ and the injective hull $Q_i$ of a simple module $L_i$ of $A$ must have the same Loewy length; otherwise there cannot be a canonical correspondence between the labelling sets for weights between $\mathcal{R}\B{R_A}$ and $\B{R_{A\op}}\op$. In fact, we show that the conditions in the statement of Theorem \ref{customthm:A} are somehow `minimal'.

\begin{customthm}{B}
\label{customthm:BBB}
Let $A$ be a finite-dimensional connected algebra with Loewy length $L$. Suppose that $\dim \End{A}{L_i}=1$ for every simple module $L_i$, and assume that the projective cover $P_i$ and the injective hull $Q_i$ of $L_i$ have both the same Loewy length $l_i$. If the Cartan matrix of $\mathcal{R}(R_A)$ coincides with the Cartan matrix of $(R_{A\op})\op$ (up to a `natural' permutation of rows and colums), then $P_i$ and $Q_i$ are rigid modules, and $l_i=L$ for every $i$.
\end{customthm}

Our final result shows that the algebra $R_A$ is not usually Ringel selfdual.

\begin{customthm}{C}
\label{customthm:BB}
The ADR algebra $R_A$ is Ringel selfdual if and only if $A$ is a selfinjective Nakayama algebra.
\end{customthm}
It was already proved in \cite{T} that selfinjective Nakayama algebras are Ringel selfdual, but in our setting this comes out as a special case.

\bigskip

The layout of the paper is the following. Section \ref{sec:background} contains background on quasihereditary algebras and on the ADR algebra. Section \ref{sec:theorema} is dedicated to the proof of our main result: Theorem \ref{customthm:A}. We start out this section by providing some evidence that supports the statement of Theorem \ref{customthm:A}. For the proof of this theorem, we collect a number of auxiliary results which relate the $\Delta$- and the $\nabla$-filtrations of the tilting modules over the ADR algebra, the highlight being Theorem \ref{thm:claim2}. In Section \ref{sec:furtherremarks}, we show that the regularity conditions for the algebra $A$ assumed in the statement of Theorem \ref{customthm:A} are, in a certain sense, minimal. This is attained by comparing several multiplicities for the algebras $\mathcal{R}\B{R_A}$ and $\B{R_{A\op}}\op$, and culminates with the proof of Theorem \ref{customthm:BBB}. In Section \ref{sec:last}, we discuss Ringel selfduality for the ADR algebra and prove Theorem \ref{customthm:BB}.

\section{Background}
\label{sec:background}
In this section we give some background on quasihereditary algebras and on the ADR algebra.

Throughout this paper the letters $B$ and $A$ shall denote arbitrary Artin algebras over some underlying commutative artinian ring $K$. All the modules will be finitely generated left modules. The notation $\Mod{B}$ will be used for the category of (finitely generated) $B$-modules. 

The case when $K$ is a field is perhaps the most significant one. In this situation, $A$ is just a finite-dimensional $K$-algebra, and the modules are finite-dimensional. For the general case, the technology works the same, and details may be found in the text book \cite{MR1476671}, Chapter II. 

\subsection{Quasihereditary algebras}
Given an Artin algebra $B$, we may label the isomorphism classes of simple $B$-modules by the elements of a finite poset $(\Phi , \sqsubseteq )$. Denote the simple $B$-modules by $L_i$, $i \in \Phi$, and use the notation $P_i$ (resp.~$Q_i$) for the projective cover (resp.~injective hull) of $L_i$. 

Let $\Delta\B{i}$ be the largest quotient of $P_i$ whose composition factors are all of the form $L_j$, with $j\sqsubseteq i$, and call $\Delta\B{i}$ the \emph{standard module} with label $i\in\Phi$. Dually, denote the \emph{costandard module} with label $i$ by $\nabla\B{i}$, i.e.~let $\nabla\B{i}$ be the largest submodule of $Q_i$ with all composition factors of the form $L_j$, with $j\sqsupseteq i$. 

Denote by $\Delta$ (resp.~$\nabla$) the set of all standard modules (resp.~costandard modules). Given a class of modules $\Theta$, let $\mathcal{F}\B{\Theta}$ be the category of all $B$-modules which have a \emph{$\Theta$-filtration}, that is, a filtration whose factors are isomorphic to modules in $\Theta$.

The notation $[M:L]$ will be used for the multiplicity of a simple module $L$ in the composition series of $M$. In a similar manner, $(M:\Delta (i))$ shall denote the multiplicity of $\Delta (i)$ in a $\Delta$-filtration of a module $M$ in $\mathcal{F}(\Delta)$. Define $(M:\nabla (i))$, $M \in \mathcal{F}(\nabla)$, in the same way.
\begin{defi}
We say that $\B{B,\Phi, \sqsubseteq}$ is \emph{quasihereditary} if the following hold for every $i \in \Phi$:
\begin{enumerate}
\item $[\Delta (i): L_i]=1$;
\item $P_i \in \mathcal{F}\B{\Delta}$;
\item $\B{P_i: \Delta\B{i}}=1$, and $\B{P_i: \Delta\B{j}}\neq 0 \Rightarrow j\sqsupseteq i$.
\end{enumerate} 
\end{defi}

An algebra $\B{B,\Phi, \sqsubseteq}$ is quasihereditary if and only if $\B{B\op,\Phi, \sqsubseteq}$ is quasihereditary. The standard modules and the costandard modules have striking homological properties. The following is well known:
\begin{enumerate}
\item if $\Ext{B}{1}{\Delta (i)}{\Delta (j)}\neq 0$, then $i\sqsubset j$;
\item $\Ext{B}{1}{\Delta (i)}{\nabla (j)} =0$ for all $i,j \in \Phi$.
\end{enumerate}

\subsubsection{Ringel duality}
In \cite{last} Ringel introduced the concept of \emph{characteristic tilting module} over a quasihereditary algebra. This is a multiplicity free $B$-module $T$, satisfying $\mathcal{F}\B{\Delta} \cap \mathcal{F}\B{\nabla} =\Add{T}$. Here for a module $M$, we denote by $\Add{M}$ the full subcategory of $\Mod{B}$ consisting of all modules isomorphic to a direct summand of a finite direct sum of copies of $M$.

It is common to refer to a module in  $\mathcal{F}\B{\Delta} \cap \mathcal{F}\B{\nabla}$ as a \emph{tilting module}. The indecomposable tilting modules are in bijection with the elements of $\Phi$. We write $T=\bigoplus_{i \in \Phi} T\B{i}$ -- the indecomposable summands $T\B{i}$ are characterised by the following result.
\begin{lem}[\cite{last}]
\label{lem:newone}
Let $\B{B, \Phi, \sqsubseteq}$ be an arbitrary quasihereditary algebra. For every $i$ in $\Phi$ there is a short exact sequence
\[
\begin{tikzcd}[ampersand replacement=\&]
0 \arrow{r} \& \Delta\B{i} \arrow{r}{\phi} \& T\B{i} \arrow{r} \& X\B{i} \arrow{r} \& 0
\end{tikzcd},
\]
with $\phi$ a left minimal $\mathcal{F}\B{\nabla}$-approximation of $\Delta\B{i}$ and with $X\B{i}$ a module lying in $\mathcal{F}\B{\{\Delta \B{j} : j \sqsubset i\}}$. 
\end{lem}

The endomorphism algebra $\END{B}{T}\op$ is known as the \emph{Ringel dual} of $B$, and shall be denoted by $\mathcal{R}\B{B}$. Let $P_i'$ be the projective indecomposable $\mathcal{R}\B{B}$-module $\Hom{B}{T}{T\B{i}}$ and let $L_i'$ be its top. According to \cite{last}, the algebra $\mathcal{R}\B{B}$ is quasihereditary with respect to the poset $(\Phi , \sqsubseteq \op )$, and moreover $\mathcal{R}\B{\mathcal{R}\B{B}}$ is Morita equivalent to $B$. Denote the standard and the costandard $\mathcal{R}\B{B}$-modules by $\Delta ' \B{i}$ and $\nabla ' \B{i}$.

\subsection{The ADR algebra as an ultra strongly quasihereditary algebra}
Fix an Artin algebra $A$. Given a module $M$ in $\Mod{A}$, we shall denote its \emph{Loewy length} by $\LL{M}$. Let $A$ have Loewy length $L$ (as a left module). We want to study the basic version of the endomorphism algebra of $\bigoplus_{j=1}^L A/ \RAD{j}{A}$.

For this, let $\{P_1, \ldots, P_n\}$ be a complete set of projective indecomposable $A$-modules and let $l_i$ be the Loewy length of $P_i$. Define 
\[
G=G_A:=\bigoplus_{i=1}^{n} \bigoplus_{j=1}^{l_i} P_i/ \RAD{j}{P_i}.
\]

The \emph{Auslander--Dlab--Ringel algebra of $A$} (ADR algebra of $A$) is defined as
\[
R_A:=\End{A}{G}\op .
\]
Observe that the functor $\Tom{A}{G}{-}:\Mod{A} \longrightarrow\Mod{R_A}$ is fully faithful as $G$ is a generator of $\Mod{A}$. 

The projective indecomposable $R_A$-modules are given by
\[
P_{i,j}:=\Hom{A}{G}{P_i/ \RAD{j}{P_i}},
\]
for $1 \leq i \leq n$, $1 \leq j \leq l_i$.

Denote the simple quotient of $P_{i,j}$ by $L_{i,j}$ and define
\begin{equation}
\label{eq:posetadr}
\Lambda := \{ (i,j):\, 1 \leq i \leq n, \, 1 \leq j \leq l_i \},
\end{equation}
so that $\Lambda$ labels the simple $R_A$-modules. Define a partial order, $\unlhd$, on $\Lambda$ by
\begin{equation}
\label{eq:guentanaborrabo}
(i,j) \lhd (k,l) \Leftrightarrow j> l.
\end{equation}

It turns out that the ADR algebra $R_A$ is a quasihereditary algebra with respect to the poset $\B{\Lambda, \unlhd}$, and its quasihereditary structure is specially neat.

Let us clarify the previous assertion. Following Ringel (\cite{RingelIyama}), a quasihereditary algebra $\B{B, \Phi, \sqsubseteq}$ is said to be \emph{right strongly quasihereditary} if $\Rad{\Delta\B{i}} \in \mathcal{F}\B{\Delta}$ for all $i \in \Phi$. This property holds if and only if the category $\mathcal{F}\B{\Delta}$ is closed under submodules (see \cite{yey}, \cite[Lemma $4.1$*]{MR1211481} and \cite[Appendix]{RingelIyama}). 

Let $(B, \Phi , \sqsubseteq )$ be an arbitrary quasihereditary algebra, as before. Additionally, suppose that $B$ satisfies the following two conditions:
\begin{description}
\item[(A1)\label{item:A1n}] $\Rad{\Delta\B{i}} \in \mathcal{F}\B{\Delta}$ for all $i \in \Phi$ (that is, $B$ is right strongly quasihereditary);
\item[(A2)\label{item:A2n}] $Q_i \in \mathcal{F}\B{\Delta}$ for all $i \in \Phi$ such that $\Rad{\Delta\B{i}}=0$.
\end{description}
We call these algebras \emph{right ultra strongly quasihereditary} algebras (RUSQ algebras, for short). The algebra $(B, \Phi , \sqsubseteq )$ is said to be a \emph{left ultra strongly quasihereditary algebra} (LUSQ) if the quasihereditary algebra $(B\op, \Phi , \sqsubseteq )$ is RUSQ.
\begin{rem}
It was proved in \cite[§$2.5.1$]{thesis} that the definition of RUSQ algebra given in \cite{MR3510398} is equivalent to the one above.
\end{rem}

According to \cite[§4]{MR3510398}, $(R_A, \Lambda, \unlhd)$ is a RUSQ algebra. The ADR algebra is the prototype of a RUSQ algebra.
\begin{thm}[{\cite[§4]{MR3510398}}]
\label{prop:standard}
The algebra $(R_A, \Lambda, \unlhd)$ is a RUSQ algebra. If $Q_i$ is the injective $A$-module with simple socle isomorphic to $\Top{P_i}$, then the following identity holds
\begin{equation*}
\label{eq:specialinj}
T\B{i,1}=Q_{i,l_i}=\Hom{A}{G}{Q_i}.
\end{equation*}
\end{thm}

\subsubsection{Properties of RUSQ algebras}
Our results about the ADR algebra rely on its properties as a RUSQ algebra. Next, we outline the main features of RUSQ algebras.

Let $\B{B,\Phi, \sqsubseteq }$ be a RUSQ algebra. It is always possible to label the elements in $\Phi$ as
\[
\Phi = \{(i,j):\, 1 \leq i \leq n ,\, 1 \leq j \leq l_i \},
\]
for certain $n,l_i \in \Zp$, so that $[\Delta\B{k,l}:L_{i,j}]\neq 0$ implies that $k=i$ and $j\geq l$ (see \cite[§5]{MR3510398}). We shall always assume that the elements in $\Phi$ are labelled in such a way. The labelling poset $\B{\Lambda , \unlhd}$ for the ADR algebra defined in \eqref{eq:posetadr} and \eqref{eq:guentanaborrabo} is compatible with this.

The following theorem summarises the main properties of the RUSQ algebras.

\begin{thm}[{\cite[§5]{MR3510398}}]
\label{thm:newprop}
Let $\B{B,\Phi, \sqsubseteq }$ be a RUSQ algebra. The following hold:
\begin{enumerate}
\item $\mathcal{F}\B{\Delta}$ is closed under submodules;
\item $\Rad{\Delta\B{i,j}}=\Delta \B{i,j+1}$ for $j < l_i$, and $\Delta\B{i,l_i}=L_{i,l_i}$;
\item each $\Delta\B{i,j}$ is uniserial and has composition factors $L_{i,j}, \ldots, L_{i,l_i}$, ordered from the top to the socle;
\item $Q_{i,l_i} \cong T\B{i,1}$ and $T\B{i,j+1} \subseteq T\B{i,j}$ for $j < l_i$;
\item $T\B{i,l_i}\cong \nabla\B{i,l_i}$, and for $j < l_i$ we have $T\B{i,j}/T\B{i,j+1}\cong \nabla\B{i,j}$ and $Q_{i,j}\cong T\B{i,1}/T\B{i,j+1}$;
\item for $M\in \mathcal{F}\B{\Delta}$, the total number of standard modules appearing in a $\Delta$-filtration of $M$ is given by $\sum_{i=1}^n [M:L_{i,l_i}]$;
\item a module $M$ belongs to $\mathcal{F}\B{\Delta}$ if and only if $\Soc{M}$ is a (finite) direct sum of modules of type $L_{i,l_i}$.
\end{enumerate}
\end{thm}

RUSQ algebras are well behaved with respect to Ringel duality.
\begin{thm}[{\cite[§6]{MR3510398}}]
\label{thm:lastonebynow}
Let $\B{B,\Phi, \sqsubseteq }$ be a RUSQ algebra. Then $\B{\mathcal{R}\B{B},\Phi, \sqsubseteq\op}$ is a LUSQ algebra. The costandard module $\nabla ' \B{i,1}$ is isomorphic to $L'_{i,1}$, and for $j>1$ we have
\[
\nabla ' \B{i,j-1} \cong \nabla ' \B{i,j} / L_{i,j}'.
\]
Each $\nabla ' \B{i,j}$ is uniserial and has composition factors $L_{i,1}', \ldots, L_{i,j}'$, ordered from the top to the socle. Moreover, $P_{i,1}' \cong T'\B{i,l_i}$ and $P_{i,j+1}' \subseteq P_{i,j}'$ for $j < l_i$. This gives rise to a filtration with factors $P_{i,l_i}' \cong \Delta'\B{i,l_i}$ and $P_{i,j}'/P_{i,j+1}'\cong \Delta ' \B{i,j}$ for $j < l_i$.
\end{thm}

\section{Theorem \ref{customthm:A}}
\label{sec:theorema}
The goal of this section is to prove Theorem \ref{customthm:A} stated in the Introduction. In order to attain this, we investigate in detail the $\Delta$-filtrations of the tilting modules over the ADR algebra $R_A$. 
\subsection{Motivation for Theorem \ref{customthm:A}}
\label{subsec:motivation}
Given an Artin algebra $A$, define
\[
C=C_A:=\bigoplus_{i=1}^n\bigoplus_{j=1}^{\LL{Q_i}} \SOC{j}{Q_i}.
\]
This is a cogenerator of $\Mod{A}$. Set $S_A:= \End{A}{C}\op$. It turns out that the algebras $S_A$ and $\mathcal{R}\B{R_A}$ have a very similar structure. In fact, the statement of Theorem \ref{customthm:A} can be loosely rephrased as: the algebras $S_A$ and $\mathcal{R}\B{R_A}$ are isomorphic provided that $A$ is ``nice enough". Before delving into the technical results necessary to prove Theorem \ref{customthm:A}, we will try to illustrate (informally) why the algebras $S_A$ and $\mathcal{R}\B{R_A}$ should be related.

In this setting, $\mathcal{R}\B{R_A}$ is equal to $\End{R_A}{T}\op$, where $T=\bigoplus_{i=1}^n\bigoplus_{j=1}^{l_i}T(i,j)$. This algebra is quasihereditary with respect to $\B{\Lambda, \unlhd \op}$, where $\B{\Lambda, \unlhd}$ is the poset associated with the ADR algebra described in \eqref{eq:posetadr} and \eqref{eq:guentanaborrabo}. According to Theorem~\ref{thm:lastonebynow}, $\B{\mathcal{R}\B{R_A},\Lambda, \unlhd \op}$ is a LUSQ algebra.

Turning the attention to the algebra $S_A$, we have that
\begin{align*}
S_A & =\End{A}{\bigoplus_{i=1}^n\bigoplus_{j=1}^{\LL{Q_i}} \SOC{j}{Q_i}} \op \\
& =\End{A}{D\B{\bigoplus_{i=1}^n\bigoplus_{j=1}^{\LL{Q_i}}P_i^{A\op}/ \RAD{j}{P_i^{A\op}}}} \op \\
& \cong \End{A\op}{G_{A\op}} =(R_{A\op})\op,
\end{align*}
where $D$ is the standard duality and $P^{A\op}_i$ denotes the projective indecomposable $A\op$-module $D\B{Q_i}$. To avoid ambiguity, denote the poset corresponding to the ADR algebra $R_{A\op}$ of $A\op$ by $\B{\Lambda_{A\op}, \trianglelefteq_{A\op}}$ and represent its elements by $[i,j]$. Note that $S_A$ is quasihereditary, as $R_{A\op}$ is. To be precise, $\B{S_A,\Lambda_{A\op}, \trianglelefteq_{A\op}}$ is a LUSQ algebra since $\B{R_{A\op},\Lambda_{A\op}, \trianglelefteq_{A\op}}$ is RUSQ.

So both $\B{\mathcal{R}\B{R_A},\Lambda, \unlhd \op}$ and $\B{S_A,\Lambda_{A\op}, \trianglelefteq_{A\op}}$ are LUSQ algebras. We take this analogy further by comparing the posets
\begin{align*}
\B{\Lambda, \trianglelefteq \op},\,\, & \Lambda= \{ (i,j):\, 1 \leq i \leq n, \, 1 \leq j \leq l_i=\LL{P_i} \}, \\
\B{\Lambda_{A\op}, \trianglelefteq_{A\op}},\,\, & \Lambda_{A\op}= \{ [i,j]:\, 1 \leq i \leq n, \, 1 \leq j \leq \LL{P^{A\op}_i}=\LL{Q_i} \}. 
\end{align*}
For the algebras $\mathcal{R}\B{R_A}$ and $S_A$ to be isomorphic they must have the same number of simple modules, i.e.~the sets $\Lambda$ and $\Lambda_{A\op}$ must have the same cardinality. It seems then reasonable to require that $\LL{P_i}=\LL{Q_i}$, for all $1 \leq i \leq n$. 

Ideally, an isomorphism between $\mathcal{R}\B{R_A}$ and $S_A$ would somehow preserve the orders $\trianglelefteq \op$ and $\trianglelefteq_{A\op}$ of $\Lambda$ and $\Lambda_{A\op}$, respectively. As $\mathcal{R}\B{R_A}$ and $S_A$ are LUSQ algebras, they both have uniserial costandard modules. By Theorem~\ref{thm:lastonebynow}, the costandard $\mathcal{R}\B{R_A}$-module with label $(i,j)$, $\nabla ' (i,j)$, has the following structure
\[
\begin{tikzcd}[ampersand replacement=\&, row sep =tiny]
(i,1) \arrow[dash]{d}\\
(i,2) \arrow[dash]{d}\\
\vdots \arrow[dash]{d} \\
(i,j)
\end{tikzcd}.
\]
The costandard $S_A$-module with label $[i,j]$, $\nabla^{S_A}[i,j]$, is isomorphic to the module $D(\Delta^{R_{A\op}}[i,j])$, where $\Delta^{R_{A\op}}[i,j]$ is the standard $R_{A\op}$-module with label $[i,j]$. Therefore, the submodule lattice of $\nabla^{S_A}[i,j]$ is `dual' to the submodule lattice of $\Delta^{R_{A\op}}[i,j]$. Using part 3 of Theorem \ref{thm:newprop}, we deduce that $\nabla^{S_A}[i,j]$ has the following structure
\[
\begin{tikzcd}[ampersand replacement=\&, row sep =tiny]
\left[  i,\LL{Q_i}\right]  \arrow[dash]{d} \\
\left[ i,\LL{Q_i} -1\right]  \arrow[dash]{d}\\
\vdots \arrow[dash]{d} \\
\left[ i,j\right] 
\end{tikzcd}.
\]
If we suppose that $\LL{P_i}=\LL{Q_i}=l_i$ for all $i$, then the modules $\nabla' (i,j)$ and $\nabla^{S_A}[i,l_i-j+1]$ have the same length for every $1 \leq i\leq n$, $1 \leq j \leq l_i$. The stronger assumption that $\LL{P_i}=\LL{Q_i}=L$ for all $i$, actually implies that the bijection $(i,j) \longmapsto [i,L-j+1]$ preserves the partial orders. In this case, we have
\begin{align*}
(i,j)\lhd\op(k,l) &\Leftrightarrow (i,j)\rhd(k,l)  \\
&\Leftrightarrow j < l  \\
&\Leftrightarrow L-j+ 1 > L-l+1  \Leftrightarrow [i,L-j+ 1]\lhd_{A\op}[k,L-l+1] .
\end{align*}

These observations support the assumptions and the claim of Theorem \ref{customthm:isoBlaststrong}.
\begin{customthm}{A}
\label{customthm:isoBlaststrong}
Suppose that $A$ satisfies $\LL{P_i}=\LL{Q_i}=L$ for all $i$, $1 \leq 1 \leq n$. Moreover, suppose that all projectives $P_i$ and all injectives $Q_i$ are rigid. Then
\[\mathcal{R}\B{R_A} \cong S_A \cong (R_{A\op})\op.\]
\end{customthm}

Recall that a module is \emph{rigid} if its radical series coincides with its socle series. The assumptions in the statement of Theorem \ref{customthm:isoBlaststrong} will be further discussed in Section~\ref{sec:furtherremarks}.
 
\subsection{Towards the proof of Theorem \ref{customthm:isoBlaststrong}}
\label{subsec:preliminaryresults}
Roughly speaking, the quasihereditary structure of the algebra $S_A$ depends on the socle series of the injective indecomposable $A$-modules, whereas the structure of the algebra $\mathcal{R}\B{R_A}$ depends on the filtrations
\[
0 \subset T\B{i,l_i} \subset \cdots \subset T\B{i,j} \subset \cdots \subset T\B{i,1}=Q_{i,l_i}
\]
mentioned in Theorem \ref{thm:newprop}. The results in this subsection explore the connections between these two filtrations. Furthermore, we determine the $\Delta$-filtration of the tilting modules $T(i,j)$ completely, when $A$ satisfies the relevant conditions for Theorem \ref{customthm:isoBlaststrong}.
 
\subsubsection{$\Delta$-semisimple filtrations for the ADR algebra}
\label{subsubsec:manuela}
Recall that the ADR algebra is a RUSQ algebra. Our proof of Theorem \ref{customthm:isoBlaststrong} uses the special properties of $\Delta$-filtrations of modules over RUSQ algebras.

We recall the definition of trace of a module. The \emph{trace} of $\Theta$ in a $B$-module $M$ is given by $\Tra{\Theta}{M}:=\sum_{f: \, f\in \Hom{B}{U}{M},\, U \in \Theta} \Ima{f}$. This is the largest submodule of $M$ generated by $\Theta$. 

A module $M$ is said to be \emph{$\Delta$-semisimple} if it isomorphic to a direct sum of standard modules. When the underlying algebra is a RUSQ algebra, the $\Delta$-semisimple modules are particularly well behaved.
\begin{prop}[{\cite[Corollary $3.3$, Proposition $3.8$]{arxiv}}]
\label{prop:dss1}
Let $\B{B,\Phi, \sqsubseteq}$ be a RUSQ algebra. The following hold:
\begin{enumerate}
\item every submodule of a $\Delta$-semisimple $B$-module is still $\Delta$-semisimple;
\item every module $M \in \mathcal{F}\B{\Delta}$ has a unique $\Delta$-semisimple submodule which is maximal among the class of all $\Delta$-semisimple submodules of $M$;
\item the largest $\Delta$-semisimple submodule of $M \in \mathcal{F}\B{\Delta}$ is given by $\delta\B{M}:=\Tra{\Delta}{M}$, and moreover $M/\delta\B{M}$ lies in $\mathcal{F}\B{\Delta}$.
\end{enumerate}
\end{prop}
Write $\delta:=\Tra{\Delta}{-}$ and let $\delta_0$ be the zero functor in $\Mod{B}$. For $i \geq 1$ and $M$ in $\Mod{B}$, define $\delta_{i+1}\B{M}$ as the module satisfying the identity $\delta_{i+1}\B{M}/\delta_{i}\B{M}= \delta \B{M/ \delta_{i}\B{M}}$. Note that $\delta_1=\delta$.

Let $\B{B,\Phi, \sqsubseteq}$ be a RUSQ algebra. Using Proposition \ref{prop:dss1}, we deduce that every module $M$ in $\mathcal{F}\B{\Delta}$ has a special filtration with proper inclusions
\[
0 \subset \delta\B{M} \subset \cdots \subset \delta_m \B{M}=M
\]
whose factors are $\Delta$-semisimple modules. This is the \emph{$\Delta$-semisimple filtration} of $M$. The integer $m$ is called the \emph{$\Delta$-semisimple length} of $M$. We write $\dssl{M}=m$.

The operators $\delta_i$ can be regarded as subfunctors of the identity functor $1_{\Mod{B}}$. In fact, the functors $\delta_i$ are left exact subfunctors of $1_{\Mod{B}}$, or in other words, they satisfy $\delta_i(N)=N \cap \delta_i(M)$ for every $N$ and $M$ with $N\subseteq M$ (we refer to \cite[§$3.2$]{arxiv} for further details).
\begin{prop}[{\cite[Lemmas $3.5$, $3.11$, $3.12$]{arxiv}}]
\label{prop:arxiv}
Let $\B{B,\Phi, \sqsubseteq}$ be a RUSQ algebra. Then the functors $\delta_i$ satisfy $\delta_i(N)=N \cap \delta_i(M)$ for every $N$ and $M$ with $N\subseteq M$. In particular, $\delta_{i}\circ \delta_{j}=\delta_{\min\{i,j\}}$. Moreover, the following hold for $M$ in $\mathcal{F}(\Delta)$:
\begin{enumerate}
\item if $i\leq \dssl{M}$, then $\dssl{\B{M/\delta_i\B{M}}}=\dssl{M}-i$;
\item if $N$ is a submodule of $M$, then $\dssl{N} \leq \dssl{M}$;
\item if $i\leq \dssl{M}$ then $\delta_i\B{M}$ is the largest $\Delta$-filtered submodule $N$ of $M$ such that $\dssl{N}=i$.
\end{enumerate}
\end{prop}
When the underlying RUSQ algebra is the ADR algebra, there is further information about the $\Delta$-semisimple filtrations.

\begin{thm}[{\cite[Lemma $4.3$, Theorem $4.4$]{arxiv}}]
\label{thm:socdelta}
Let $M$ be in $\Mod{A}$. Then the $R_A$-module $N:=\Hom{A}{G}{M}$ lies in $\mathcal{F}\B{\Delta}$ and the socle series of $M$ determines the $\Delta$-semisimple filtration of $N$. More precisely,
\[
\delta_i \B{N} = \Hom{A}{G}{\SOC{i}{M}},
\]
for all $i$, and $\dssl{N}=\LL{M}$. Moreover, if $\SOC{i}{M}/\SOC{i-1}{M}=\bigoplus_{\theta \in \Theta} L_{x_{\theta}}$, then
\[
\delta_{i}\B{N}/\delta_{i-1}\B{N}= \bigoplus_{\theta \in \Theta} \Delta\B{x_{\theta}, i}.
\]
\end{thm}

Quotients of socle series and quotients of $\Delta$-semisimple filtrations will occur frequently, and to reduce necessary symbols, we will use the following notation: for any module $M$, any module $N$ in $\mathcal{F}\B{\Delta}$ and $i \geq 1$ we write
\[\USOC{i}{M}:= \SOC{i}{M}/ \SOC{i-1}{M} \quad\text{and}\quad \underline{\delta}_{i}\B{N}:= \delta_{i}\B{N}/\delta_{i-1}\B{N}.\]
\subsubsection{Preliminary results}
We now investigate the structure of the $R_A$-modules for algebras $A$ such that $\LL{P_i} =L$ for all $i$. Then, the minimal elements in the poset $(\Lambda, \unlhd)$ are precisely all $(k,L)$ for $1 \leq k \leq n$. Soon we will also study the situation when all $P_i$ are rigid, that is, when their radical series and their socle series coincide.

\begin{lem}
\label{lem:pink}
Let $A$ be such that $\LL{P_i}=L$ for all $i$, $1 \leq i \leq n$. Then
\[T\B{k,L}=L_{k,L}=\Delta\B{k,L},\]
for $1 \leq k \leq n$.
\end{lem}
\begin{proof}
Recall the description of the indecomposable tilting modules in Lemma \ref{lem:newone}. The composition factor $L_{k,l}$ has multiplicity one in both $\Delta\B{k,l}$ and $T\B{k,l}$, and the composition factors of these two modules are of the form $L_{i,j}$, with $(i,j)\unlhd (k,l)$. The lemma follows from the fact that each $(k,L)$ is a minimal element in $( \Lambda , \unlhd)$.
\end{proof}

\begin{lem}
\label{lem:karinlastt}
Let $A$ be such that $\LL{P_i}=L$ for all $i$, $1 \leq i \leq n$. Let $M$ be in $\mathcal{F}\B{\Delta}$ and let $l$ be the smallest integer such that $(M:\Delta(k,l))\neq 0$ for some $k$. Then $\dssl{M}\leq L-l+1$.
\end{lem}
\begin{proof}
We use downwards induction on $l$. If $l=L$, then all the factors in a $\Delta$-filtration of $M$ are of the form $\Delta (k,L)$ for some $k$. Since two standard modules of this form have no nontrivial extensions, it follows that $M$ is a direct sum of standard modules, so that $M=\delta_1 (M)$ and $\dssl{M}=1$. 

Let $l <L$. Assume the claim holds for modules $N$ where a minimal $l'$ with $(N:\Delta (k,l'))\neq 0$ is such that $l < l' \leq L$. Consider $M$ as in the statement of the lemma. There is an exact sequence
\[
\begin{tikzcd}[ampersand replacement=\&]
0 \arrow{r} \& \delta_1\B{M} \arrow{r} \& M \arrow{r} \& M/ \delta_1(M) \arrow{r} \& 0 
\end{tikzcd}.
\]
Note that $\Ext{R_A}{1}{\Delta\B{k,l}}{\Delta\B{i,j}}=0$ for any $i$, $k$ and $j\geq l$. As a consequence, any $\Delta (k,l)$ which occurs in a $\Delta$-filtration of $M$ must occur in $\delta_1 (M)$ since it has no nontrivial extensions with any other standard module which may appear in a $\Delta$-filtration of $M$. Therefore, a minimal $l'$ with $(M/ \delta_1 (M):\Delta (k,l'))\neq 0$ satisfies $l<l'$. By the induction hypothesis, $M/ \delta_1(M)$ has $\Delta$-semisimple length at most $L-l'+1$. Proposition \ref{prop:arxiv} implies that $\dssl{M}\leq L-l'+1+1\leq L-l+1$. 
\end{proof}

\begin{prop}
\label{prop:tb1}
Let $A$ be such that $\LL{P_i}=L$ for all $i$, $1 \leq i \leq n$. Then, for every $(k,l)$ in $\Lambda$, we have
\[ T\B{k,l} \subseteq \Hom{A}{G}{\SOC{L-l+1}{Q_k}}= \delta_{L-l+1}\B{Q_{k,L}}= \delta_{L-l+1}\B{T\B{k,1}}.
\]
\end{prop}
\begin{proof}
According to Theorem \ref{prop:standard}, we have that $T\B{k,1}=Q_{k,L}=\Hom{A}{G}{Q_k}$. Theorem \ref{thm:socdelta} implies that
\[\Hom{A}{G}{\SOC{L-l+1}{Q_k}}= \delta_{L-l+1}\B{Q_{k,L}}=\delta_{L-l+1}\B{T\B{k,1}}.\]
By part 4 of Theorem \ref{thm:newprop}, it follows that $T\B{k,l}\subseteq T\B{k,1}$. By Lemma \ref{lem:karinlastt},  $\dssl{T\B{k,l}} \leq L-l+1$. According to part 3 of Proposition \ref{prop:arxiv}, this shows that $T\B{k,l}$ is contained in $\delta_{L-l+1}\B{T\B{k,1}}$.
\end{proof}

By Proposition \ref{prop:tb1}, if all the projectives in $\Mod{A}$ have the same Loewy length, then $T\B{k,l}$ is a submodule of $\delta_{L-l+1}\B{Q_{k,l}}$ for every $(k,l)$ in $\Lambda$. If additionally all projectives $P_i$ are rigid, then a $\Delta$-filtration of $T\B{k,l}$ has the same number of factors as a $\Delta$-filtration of $\delta_{L-l+1}\B{Q_{k,L}}$.
\begin{prop}
\label{prop:almostlastt}
Suppose that $A$ satisfies $\LL{P_i}=L$ for all $i$, $1 \leq i \leq n$. Assume that the projectives $P_i$ are rigid. Then the monic
\[
\begin{tikzcd}[ampersand replacement=\&]
\Hom{R_A}{P_{i,L}}{T\B{k,l}} \arrow[hook]{r} \& \Hom{R_A}{P_{i,L}}{\delta_{L-l+1}\B{Q_{k, L}}} 
\end{tikzcd}
\]
induced by the inclusion
\[
T\B{k,l} \subseteq \delta_{L-l+1}\B{Q_{k, L}}
\]
is an isomorphism. In particular, the modules $T\B{k,l}$ and $\delta_{L-l+1}\B{Q_{k, L}}$ are filtered by the same number of standard modules.
\end{prop}
\begin{proof}
Recall that the $K$-module $\Hom{R_A}{P_{i,L}}{\delta_{L-l+1}\B{Q_{k,L}}}$ is isomorphic to the module $\Hom{A}{P_i}{\SOC{L-l+1}{Q_k}}$ via the functor  $\Hom{A}{G}{-}$. So consider a morphism $f:P_i \longrightarrow \SOC{L-l+1}{Q_k}$ and the corresponding map $f_*=\Hom{A}{G}{f}$. To prove the statement, we must show that the image of $f_*$ is contained in $T(k,l)$. According to \cite[Lemma $5.7$]{MR3510398}, it is enough to prove that all the composition factors of $\Ima{f_*}$ are of the form $L_{x,y}$ with $(x,y) \ntriangleright (k,l)$, that is, with $y \geq l$. The module $\SOC{L-l+1}{Q_k}$ has Loewy length $L-l+1$, hence $\RAD{L-l+1}{P_i} \subseteq \Ker{f}$. Now, $P_i$ is rigid and therefore $\RAD{L-l+1}{P_i}=\SOC{l-1}{P_i}$. Thus
\[
\delta_{l-1} (P_{i,L})=\Hom{A}{G}{\RAD{L-l+1}{P_i}} \subseteq \Ker{f_*}.
\]
By Theorem \ref{thm:socdelta} , the quotient $P_{i,L}/\delta_{l-1}(P_{i,L})$ is only filtered by standard modules $\Delta (s,t)$ with $t\geq l$, and hence all composition factors of $\Ima{f_*}$ are of the form $L_{x,y}$ with $y \geq l$. The last assertion in the statement of the proposition follows from part 6 of Theorem \ref{thm:newprop}.
\end{proof}

The last part of Proposition \ref{prop:almostlastt} suggests that the total number of $\Delta$-quotients is an important invariant in this setting. If $M$ is in $\mathcal{F}\B{\Delta}$, we denote the total number of $\Delta$-quotients of $M$ by $\operatorname{r} (M)$. In this context, we shall refer to $\operatorname{r}(M)$ as the \emph{rank} of $M\in \mathcal{F}(\Delta)$. By part 6 of Theorem \ref{thm:newprop}, the rank of $M\in \mathcal{F}(\Delta)$ is equal to the total number of composition factors of $M$ which are of the form $L_{k,l_k}$ as $k$ varies.

Using Proposition \ref{prop:almostlastt} it is possible to compute the $\Delta$-semi\-sim\-ple length of all tilting modules $T\B{k,l}$ in the case when all the projectives in $\Mod{A}$ are rigid and have the same Loewy length. 
\begin{lem}
\label{lem:claim1}
Suppose that $A$ satisfies $\LL{P_i}=L$ for all $i$, $1 \leq i \leq n$, and assume that the projectives are rigid. Then 
\[\dssl{T\B{k,l}}=\min\{L-l+1, \LL{Q_k}\},\]
for $(k,l) \in \Lambda$. In particular, if $\LL{Q_i}=L$ for all $i$, then $\dssl{T\B{k,l}}=L-l+1$ for all $(k,l) \in \Lambda$.
\end{lem}

In order to prove Lemma \ref{lem:claim1}, we will apply the following general principle.
\begin{lem}
\label{lem:larkascending}
Let $\B{B, \Phi, \sqsubseteq}$ be a RUSQ algebra, and let $M$ be in $\mathcal{F}\B{\Delta}$. Assume that $N$ is a submodule of $M$ satisfying $\operatorname{r}(N)=\operatorname{r}(M)$. There is a canonical monic $\underline{\delta}_i (N) \longrightarrow \underline{\delta}_i (M)$ for every $i$. If $\underline{\delta}_i (M) \cong \bigoplus_{\omega \in \Omega_i} \Delta (x_{\omega}, y_{\omega})$ (for some index set $ \Omega_i$), then
\[
\underline{\delta}_i (N) \cong \bigoplus_{\omega \in \Omega_i} \Delta (x_{\omega}, y_{\omega}')
\]
where $y_{\omega} \leq y_{\omega}'$ for all $\omega$. In particular, $\dssl{M}=\dssl{N}$.
\end{lem}
\begin{rem}
\label{rem:larkascending}
Let $\B{B, \Phi, \sqsubseteq}$ be a RUSQ algebra. Recall from §\ref{subsubsec:manuela} that the functors $\delta_i$ satisfy $\delta_i(N)=N \cap \delta_i(M)$ for every $N$ and $M$ with $N\subseteq M$. This implies that there is a well-defined monomorphism $N/\delta_i(N) \longrightarrow M/ \delta_i (M)$, mapping $n + \delta_i(N)$ to $n + \delta_i(M)$.
\end{rem}
\begin{proof}[Proof of Lemma \ref{lem:larkascending}]
Note that $\delta_i (N) \subseteq \delta_i (M)$. Using part 6 of Theorem \ref{thm:newprop}, we deduce that $\operatorname{r}(\delta_i(N)) \leq \operatorname{r}(\delta_i(M))$. By Remark \ref{rem:larkascending}, we also conclude that $\operatorname{r}(N/\delta_i(N)) \leq \operatorname{r}(M/\delta_i(M))$. Since
\begin{align*}
\operatorname{r} \B{M} &= \operatorname{r} \B{\delta_i \B{M}}  + \operatorname{r} \B{M/ \delta_i \B{M}} \\
& \geq \operatorname{r} \B{\delta_i \B{N}}  + \operatorname{r} \B{N/ \delta_i \B{N}} \\
& = \operatorname{r} (N) =\operatorname{r} (M),
\end{align*}
it follows that $\operatorname{r} \B{\delta_i \B{M}} =\operatorname{r} \B{\delta_i \B{N}} $ and $\operatorname{r} \B{M/\delta_i \B{M}} =\operatorname{r} \B{N/\delta_i \B{N}}$ for all $i$. As a consequence, we deduce that $\operatorname{r} \B{\underline{\delta}_i \B{M}} =\operatorname{r} \B{\underline{\delta}_i \B{N}} $ for all $i$. Therefore, the canonical monic in Remark \ref{rem:larkascending} restricts to a monic $\underline{\delta}_i \B{N}\longrightarrow\underline{\delta}_i \B{M}$ between $\Delta$-semisimple modules which must satisfy the claim in the statement of the lemma.
\end{proof}

\begin{proof}[Proof of Lemma \ref{lem:claim1}]
By Proposition \ref{prop:almostlastt}, $T\B{k,l}$ is contained in $\delta_{L-l+1}\B{Q_{k, L}}$ and they have the same rank. According to Lemma \ref{lem:larkascending}, we must have $\dssl{T\B{k,l}}=\dssl{\delta_{L-l+1}\B{Q_{k, L}}}$. Using Theorem \ref{thm:socdelta}, we conclude that the $\Delta$-semisimple length of $Q_{k,L}=\Hom{A}{G}{Q_k}$ is $\LL{Q_k}$. From Proposition \ref{prop:arxiv}, we deduce that the $\Delta$-semisimple length of $\delta_{L-l+1}\B{Q_{k, L}}$ is given by $\min\{L-l+1, \LL{Q_k}\}$.
\end{proof}

We now go one step further and fully describe, in Theorem \ref{thm:claim2}, the $\Delta$-semisimple filtration of $T\B{k,l}$ in terms of the socle series of $Q_k$ in the case when all the projectives in $\Mod{A}$ are rigid and have the same Loewy length. For the proof of Theorem \ref{thm:claim2}, the following result will be useful.

\begin{lem}
\label{lem:superpinklast}
Let $\B{B, \Phi, \sqsubseteq}$ be a RUSQ algebra, and let $M$ be in $\mathcal{F}\B{\Delta}$, with $\dssl{M}=m\geq 2$. Suppose that $2 \leq i \leq m$, and let $\pi$ be a split epic mapping the $\Delta$-semisimple module $\delta_i\B{M}/ \delta_{i-1}\B{M}$ onto a summand $\Delta\B{k,l}$. Then the epic
\[
\begin{tikzcd}[ampersand replacement=\&]
\delta_i\B{M}/ \delta_{i-2}\B{M} \arrow[two heads]{r} \&\underline{ \delta}_i\B{M} \arrow[two heads]{r}{\pi}\&  \Delta\B{k,l}
\end{tikzcd}
\]
does not split.
\end{lem}
\begin{proof}
Denote the canonical epic $\delta_i\B{M}/ \delta_{i-2}\B{M} \longrightarrow \underline{ \delta}_i\B{M}$ by $\varpi$. Suppose, by contradiction, that $\pi \circ \varpi$ splits. Then $\delta_i\B{M}/ \delta_{i-2}\B{M} \cong \Ker{(\pi \circ \varpi)} \oplus \Delta\B{k,l}$. Note that $\Ker{\varpi} \subseteq \Ker{(\pi \circ \varpi)}$, and that $\Ker{\varpi}\cong \underline{ \delta}_{i-1}\B{M}$. As a consequence, there is a monic from $\Ker{\varpi}\oplus \Delta\B{k,l}$ to $\delta_i\B{M}/ \delta_{i-2}\B{M}$, so $\Soc{(\underline{\delta}_{i-1}\B{M})} \oplus L_{k,l_k}$ can be embedded in $\Soc{(\delta_i\B{M}/ \delta_{i-2}\B{M})}$. Using the first statement in Proposition \ref{prop:arxiv}, is easy to check that $\underline{\delta}_{i-1}\B{M}=\delta(\delta_i\B{M}/ \delta_{i-2}\B{M})$. Since $\Soc{N}=\Soc{\delta\B{N}}$ for any $B$-module $N$ (see \cite[§§$3.2.2$]{arxiv}), then $\Soc{(\delta_i\B{M}/ \delta_{i-2}\B{M})}=\Soc{(\underline{\delta}_{i-1}\B{M})}$, which leads to a contradiction.
\end{proof}

Recall that the $\Delta$-semisimple filtration of the projectives $P_{i,j}$ is determined by the radical series of $P_i$ when $P_i$ is rigid: this is a consequence of Theorem \ref{thm:socdelta}.
\begin{thm}
\label{thm:claim2}
Suppose that $A$ satisfies $\LL{P_i}=L$ for all $i$, $1\leq i \leq n$, and assume that the projectives are rigid. Let $(k,l)\in\Lambda$, and suppose that the socle layers of $Q_k$ are
\[
\USOC{i}{Q_k} \cong \bigoplus_{\omega \in \Omega^{k}_i} L_{x_{\omega}},
\]
for $i=1 , \ldots, \LL{Q_k}$. Then
\[
\underline{\delta}_i\B{T\B{k,l}} \cong \bigoplus_{\omega \in  \Omega^{k}_i} \Delta\B{x_{\omega}, l+i-1},
\]
for $i=1 , \ldots, \dssl{T\B{k,l}}$.
\end{thm}

We outline the strategy of the proof of Theorem \ref{thm:claim2}. Throughout, $l_k=L$ for all $k$, and all projectives are assumed to be rigid. Suppose that $\USOC{i}{Q_k}=\bigoplus_{\omega \in \Omega^{k}_i} L_{x_{\omega}}$. In order to prove that $\underline{\delta}_i\B{T\B{k,l}} \cong \bigoplus_{\omega \in  \Omega^{k}_i} \Delta\B{x_{\omega}, l+i-1}$, we proceed in \emph{two steps}.
\begin{enumerate}
\item We show that $\underline{\delta}_i (T(k,l)) \cong \bigoplus_{\omega \in  \Omega^{k}_i} \Delta\B{x_{\omega}, y_{\omega,l}}$ with $l+i-1 \leq y_{\omega,l} \leq L$, for all $1 \leq i \leq \dssl{T\B{k,l}}$. In particular, the rank $\operatorname{r}(\delta_{i}(T(k,l)))$ is constant on $\delta_{i}(T(k,l))$ as $l$ varies, for all $1 \leq i \leq \dssl{T\B{k,l}}$.
\item We show that $y_{\omega,l} =l+i-1$ for all $\omega \in  \Omega^{k}_i$.
\end{enumerate}

\begin{proof}[Proof of Theorem \ref{thm:claim2}]
Fix $k$ and $l$, with $1\leq k \leq n$ and $1 \leq l \leq L$.  Assume that the socle series of $Q_k$ is given by $\USOC{i}{Q_k}=\bigoplus_{\omega \in \Omega^{k}_i} L_{x_{\omega}}$. We prove the statement of the theorem by induction on $i$, $1 \leq i \leq \dssl{T\B{k,l}}$. Note that $\delta_1\B{T\B{k,l}}=\Delta \B{k,l}$. As $\Soc{Q_k}=L_k$, then $|\Omega_{1}^{k}|=1$ and $x_{\omega}=k$ for $\omega \in \Omega_{1}^{k}$, thus the claim holds trivially for $i=1$.  So let $i$ be such that $2 \leq i \leq \dssl{T\B{k,l}}$ and suppose, by induction, that
\[
\underline{\delta}_{i-1}\B{T\B{k,l}}=\bigoplus_{\omega \in \Omega_{i-1}^k} \Delta\B{x_{\omega}, l+i-2}.
\]
We wish to describe $\underline{\delta}_{i}\B{T\B{k,l}}$.
\paragraph*{Step 1}
\label{para:step1}
Recall that $T\B{k,1}=Q_{k,L}=\Hom{A}{G}{Q_k}$. By Theorem \ref{thm:socdelta}, $\underline{\delta}_i\B{T\B{k,1}}$ is isomorphic to $\bigoplus_{\omega \in  \Omega^{k}_i} \Delta\B{x_{\omega}, i}$. By Proposition \ref{prop:almostlastt}, $T\B{k,l}\subseteq \delta_{L-l+1}\B{T\B{k,1}}$ and $\operatorname{r}(T\B{k,l})=\operatorname{r}\B{\delta_{L-l+1}(T\B{k,1})}$. By Lemma \ref{lem:larkascending} (using Proposition \ref{prop:arxiv}), there is a canonical monic mapping $\underline{\delta}_i\B{T\B{k,l}} \longrightarrow \underline{\delta}_i\B{T\B{k,1}}$, and $\underline{\delta}_i\B{T\B{k,l}}$ is isomorphic to
\[
\bigoplus_{\omega \in  \Omega^{k}_i} \Delta\B{x_{\omega}, y_{\omega,l}},
\]
with $i \leq y_{\omega,l} \leq L$. We want to show that $ y_{\omega,l} \geq l+i-1$ for every $\omega$. Take $\omega' \in  \Omega^{k}_i$ so that the integer $y_{\omega ',l}$ is minimal. If $y_{\omega ',l} \leq l+i-2$, then, by induction, $(x_{\omega '},y_{\omega',l}) \not\vartriangleleft (x,y)$ for every standart module $\Delta (x,y)$ filtering the quotient $\delta_{i}(T(k,l))/\delta_{i-2}(T(k,l))$. Thus, the canonical epic
\[
\begin{tikzcd}[ampersand replacement=\&]
\delta_i\B{T\B{k,l}}/ \delta_{i-2}\B{T\B{k,l}} \arrow[two heads]{r} \& \underline{\delta}_i\B{T\B{k,l}} \arrow[two heads]{r}\&  \Delta\B{x_{\omega '},y_{\omega',l}}
\end{tikzcd}
\]
splits. This contradicts Lemma \ref{lem:superpinklast}, therefore $ y_{\omega,l} \geq l+i-1$ for every $\omega$.

\paragraph*{Step 2}
Let $U:=\bigoplus_{\omega \in  \Omega^{k}_i} \Delta\B{x_{\omega}, l+i-1}$, which is a submodule of $\underline{\delta}_i(T(k,1))$ and let $\iota:U \longrightarrow \underline{\delta}_i (T(k,1))$ be the inclusion map. Take $\nu:\delta_i(T(k,1)) \longrightarrow \underline{\delta}_i (T(k,1))$ to be the canonical epic with kernel $\delta_{i-1}(T(k,1))$ and consider a projective cover
\[
\pi: P:=\bigoplus_{\omega \in \Omega^{k}_i} P_{x_{\omega}, l+i-1} \longrightarrow U.
\]
There is a morphism $f_* :P \longrightarrow \delta_i (T(k,1))$ such that $\iota \circ \pi= \nu \circ f_*$. We claim that the image of $f_*$ is contained in $T(k,l)$. For this we need the characterisation of $T(k,l)$ in \cite[Lemma $5.7$]{MR3510398}: this is the largest submodule of $T(k,1)$ such that all composition factors are of the form $L_{x,y}$ with $y\geq l$. Recall that $f_*=\Hom{A}{G}{f}$ where
\[
f: \bigoplus_{\omega \in \Omega^{k}_i} P_{x_{\omega}}/\RAD{l+i-1}{P_{x_{\omega}}} \longrightarrow \SOC{i}{Q_k}.
\]
The image of $f$ has Loewy length at most $i$, thus $\bigoplus_{\omega \in \Omega^{k}_i} \RAD{i}{P_{x_{\omega}}}/\RAD{l+i-1}{P_{x_{\omega}}}$ is mapped to zero, and consequently
\[\Hom{A}{G}{\bigoplus_{\omega \in \Omega^{k}_i} \RAD{i}{P_{x_{\omega}}}/\RAD{l+i-1}{P_{x_{\omega}}}}\subseteq \Ker{f_*}.
\]
Using that the projective indecomposable $A$-modules are rigid, together with Theorem \ref{thm:socdelta}, this can be rewritten as $\delta_{l-1}(P) \subseteq \Ker{f_*}$. Therefore, the image of $f_*$ is a quotient of $P/\delta_{l-1}(P)$. From Theorem \ref{thm:socdelta}, we deduce that $\Ima{f_*}$ has only composition factors of the form $L_{x,y}$ with $y\geq l$. Hence $\Ima{f_*}\subseteq T(k,l)$ by a previous observation. As a consequence,
\[
\Ima{f_*} \subseteq T(k,l) \cap \delta_i \B{T\B{k,1}} = \delta_i \B{T\B{k,l}}.
\]
Using that $\delta_{i-1}(N)=N \cap \delta_{i-1}(M)$ for every $N$ and $M$ with $N\subseteq M$, one deduces that restriction of $\nu$ to $\delta_i \B{T\B{k,l}}$ factors through the canonical monic $\underline{\delta}_{i}(T(k,l)) \longrightarrow \underline{\delta}_{i}(T(k,1))$. Hence $\Ima{(\nu \circ f_*)}$ can be embedded in $\underline{\delta}_{i}(T(k,l))$. Since $\iota \circ \pi= \nu \circ f_*$, it follows that $\underline{\delta}_{i}(T(k,l))$ has a submodule isomorphic to $U$. By the conclusion of \nameref{para:step1}, we must have $\underline{\delta}_{i}(T(k,l))\cong U$.\end{proof}

\subsection{Proof of Theorem \ref{customthm:isoBlast}}

We finally prove the main result of this paper.

\begin{customthm}{A}
\label{customthm:isoBlast}
Suppose that $A$ satisfies $\LL{P_i}=\LL{Q_i}=L$ for all $i$, $1 \leq i \leq n$. Moreover, suppose that all projectives $P_i$ and all injectives $Q_i$ are rigid. Then
\[\mathcal{R}\B{R_A} \cong S_A \cong (R_{A\op})\op.\]
\end{customthm}

The two key ingredients for the proof of Theorem \ref{customthm:isoBlast} (Propositions \ref{prop:karinclaim1} and \ref{prop:karinclaim2}) rely on the description of the $\Delta$-semisimple filtration of the tilting modules given in Theorem \ref{thm:claim2}.

Recall that the underlying algebra $A$ is an Artin $K$-algebra. Therefore the ADR algebra $R_A$ is also an Artin $K$-algebra, and $\Hom{R_A}{X}{Y}$ lies in $\Mod{K}$ for $X$ and $Y$ in $\Mod{R_A}$.
\begin{prop}
\label{prop:karinclaim1}
Suppose that $A$ satisfies $\LL{P_i}=\LL{Q_i}=L$ for all $i$, $1 \leq i \leq n$, and assume that all projectives $P_i$ and all injectives $Q_i$ are rigid. Then the $K$-modules $\Tom{R_A}{T(k,l)}{T(i,j)}$ and $\Tom{R_A}{\delta_{L-l+1}(Q_{k,L})}{\delta_{L-j+1}(Q_{i,L})}$ have the same (Jordan--H\"older) length. 
\end{prop}
\begin{proof}
Denote the length of a module $M$ in $\Mod{K}$ by $\operatorname{l}(M)$. We will use the notation in the statement of Theorem \ref{thm:claim2} to describe the socle layers of the injective indecomposable module $Q_k$. 

We start by determining the value of $\operatorname{l}(\Tom{R_A}{T(k,l)}{T(i,j)})$. Using that the functor $\Hom{R_A}{-}{T(i,j)}$ preserves exact sequences in $\mathcal{F}\B{\Delta}$ (see \cite[Corollary 4]{last}), together with Theorem \ref{thm:claim2}, we deduce that
\[
\operatorname{l}\B{\Hom{R_A}{T\B{k,l}}{T\B{i,j}}}=\sum_{y=1}^{\dssl{T\B{k,l}}} \sum_{\omega \in \Omega_{y}^k} \operatorname{l}\B{\Hom{R_A}{\Delta\B{x_{\omega}, l+y-1}}{T\B{i,j}}}
\]
Lemma \ref{lem:claim1} implies that $\dssl{T\B{k,l}}=L-l+1$. By Theorem \ref{thm:newprop}, the module $T\B{i,j}$ is filtered by the costandard modules $\nabla\B{i,j}, \nabla\B{i,j+1}, \ldots, \nabla\B{i,L}$. Using again that $\Ext{R_A}{1}{\mathcal{F}(\Delta)}{\mathcal{F}(\nabla)} = 0$, we get
\begin{multline*}
\sum_{y=1}^{L-l+1} \sum_{\omega \in \Omega_{y}^k} \operatorname{l}\B{\Hom{R_A}{\Delta\B{x_{\omega}, l+y-1}}{T\B{i,j}}}  \\
=\sum_{y=1}^{L-l+1} \sum_{\omega \in \Omega_{y}^k} \sum_{z=j}^L \operatorname{l}\B{\Hom{R_A}{\Delta\B{x_{\omega}, l+y-1}}{\nabla\B{i,z}}}.
\end{multline*}
Note that
\begin{multline*}
\sum_{y=1}^{L-l+1} \sum_{\omega \in \Omega_{y}^k} \sum_{z=j}^L \operatorname{l}\B{\Hom{R_A}{\Delta\B{x_{\omega}, l+y-1}}{\nabla\B{i,z}}} \\
\begin{aligned}
&= \sum_{y=1}^{L-l+1} \sum_{\omega \in \Omega_{y}^k} \sum_{z=j}^L \delta_{\B{x_{\omega}, l+y-1},\B{i,z}} \operatorname{l}\B{\End{R_A}{\Delta\B{x_{\omega}, l+y-1}}}\\
&= \sum_{y=1}^{L-l+1} \sum_{\omega \in \Omega_{y}^k} \sum_{z=j}^L \delta_{\B{x_{\omega}, l+y-1},\B{i,z}} \operatorname{l}\B{\End{A}{L_{x_{\omega}}}}\\
&= \sum_{y=\max\{j-l, 0\}+1}^{L-l+1} \sum_{\omega \in \Omega_{y}^k} \delta_{x_{\omega},i} \operatorname{l}\B{\End{A}{L_{x_{\omega}}}}.\\
\end{aligned}
\end{multline*}
In here the second equality will follow from Lemma \ref{lem:seraultimo}, and the third equality follows by analysing the values taken by the Kronecker delta.

Now we calculate $\operatorname{l}(\Tom{R_A}{\delta_{L-l+1}(Q_{k,L})}{\delta_{L-j+1}(Q_{i,L})})$. Observe that
\begin{multline*}
\Hom{R_A}{\delta_{L-l+1}\B{Q_{k,L}}}{\delta_{L-j+1}\B{Q_{i,L}}} \\ 
\begin{aligned}
&= \Hom{R_A}{\Hom{A}{G}{\SOC{L-l+1}{Q_k}}}{\Hom{A}{G}{\SOC{L-j+1}{Q_i}}}\\
& \cong \Hom{A}{\SOC{L-l+1}{Q_k}}{\SOC{L-j+1}{Q_i}},\\
\end{aligned}
\end{multline*}
where the first equality follows from Proposition \ref{prop:tb1} and the second identity is due to the fact that $\Hom{A}{G}{-}$ is a fully faithful functor. Any map $f:\SOC{L-l+1}{Q_k}\longrightarrow\SOC{L-j+1}{Q_i}$ is such that $\LL{\Ima{f}} \leq L-j+1$, so $f$ must factor through the largest quotient of $\SOC{L-l+1}{Q_k}$ whose Loewy length is at most $L-j+1$. That is, $f$ factors through the module
\begin{multline*}
\SOC{L-l+1}{Q_k}/ \RAD{L-j+1}{\B{\SOC{L-l+1}{Q_k}}} \\
=\SOC{L-l+1}{Q_k}/ \SOC{\max\{L-l+1-(L-j+1),0\}}{Q_k},
\end{multline*}
where the equality follows from the rigidity of $Q_k$. So the canonical epic
\[\SOC{L-l+1}{Q_k}\longrightarrow\SOC{L-l+1}{Q_k}/ \SOC{\max\{j-l,0\}}{Q_k}\]
induces an isomorphism of $K$-modules
\begin{multline*}
\Hom{A}{\SOC{L-l+1}{Q_k}/ \SOC{\max\{j-l,0\}}{Q_k}}{\SOC{L-j+1}{Q_i}} \\ \cong\Hom{A}{\SOC{L-l+1}{Q_k}}{\SOC{L-j+1}{Q_i}}.
\end{multline*}
Notice that both $\SOC{L-l+1}{Q_k}/ \SOC{\max\{j-l,0\}}{Q_k}$ and $\SOC{L-j+1}{Q_i}$ are modules over $A/\RAD{L-j+1}{A}$. In fact, $\SOC{L-j+1}{Q_i}$ is an injective in $\Mod{(A/\RAD{L-j+1}{A})}$. Thus the restriction of $\Hom{A}{-}{\SOC{L-j+1}{Q_i}}$ to $\Mod{(A/\RAD{L-j+1}{A})}$ yields an exact functor. Therefore
\begin{multline*}
\operatorname{l}\B{\Hom{R_A}{\delta_{L-l+1}\B{Q_{k,L}}}{\delta_{L-j+1}\B{Q_{i,L}}}} \\
\begin{aligned}
&=\operatorname{l}\B{\Hom{A}{\SOC{L-l+1}{Q_k}/ \SOC{\max\{j-l,0\}}{Q_k}}{\SOC{L-j+1}{Q_i}}} \\
&=\sum_{y=\max\{j-l,0\}+1}^{L-l+1}\sum_{\omega \in \Omega_y^k} \operatorname{l}\B{\Hom{A}{L_{x_{\omega}}}{\SOC{L-j+1}{Q_i}}} \\
&=\sum_{y=\max\{j-l,0\}+1}^{L-l+1}\sum_{\omega \in \Omega_y^k} \operatorname{l}\B{\Hom{A}{L_{x_{\omega}}}{L_i}} \\
&=\sum_{y=\max\{j-l,0\}+1}^{L-l+1}\sum_{\omega \in \Omega_y^k} \delta_{x_{\omega},i} \operatorname{l}\B{\End{A}{L_{x_{\omega}}}}, 
\end{aligned}
\end{multline*}
which shows that $\Tom{R_A}{\delta_{L-l+1}(Q_{k,L})}{\delta_{L-j+1}(Q_{i,L})}$ and $\Tom{R_A}{T(k,l)}{T(i,j)}$ have the same length over $K$.
\end{proof}
\begin{lem}
\label{lem:seraultimo}
Let $\B{B, \Phi, \sqsubseteq}$ be a RUSQ algebra (over $K$). Then
\[
\End{B}{\Delta\B{i,j+1}} \cong \Hom{B}{\Delta\B{i,j+1}}{\Delta\B{i,j}} \cong \End{B}{\Delta\B{i,j}}
\]
as $K$-modules. If $B$ is the ADR algebra $R_A$ of an Artin algebra $A$ then the modules above are isomorphic to $\End{A}{L_i}$.
\end{lem}
\begin{proof}
Consider the short exact sequence in $\Mod{B}$
\[
\begin{tikzcd}[ampersand replacement=\&]
0 \arrow{r} \& \Delta\B{i,j+1} \arrow{r} \& \Delta\B{i,j} \arrow{r} \&  L_{i,j} \arrow{r} \& 0.
\end{tikzcd}
\]
By applying the functor $\Hom{B}{\Delta\B{i,j+1}}{-}$ to this exact sequence we deduce that $\End{B}{\Delta\B{i,j+1}}\cong \Hom{B}{\Delta\B{i,j+1}}{\Delta\B{i,j}}$. Using $\Hom{B}{-}{\Delta\B{i,j}}$, we get an exact sequence
\[
\begin{tikzcd}[ampersand replacement=\&, column sep=small]
0 \arrow{r} \& \End{B}{\Delta\B{i,j}} \arrow{r} \& \Hom{B}{\Delta\B{i,j+1}}{\Delta\B{i,j}} \arrow{r} \&  \Ext{B}{1}{L_{i,j}}{\Delta\B{i,j}}.
\end{tikzcd}
\]
Note that $\Ext{B}{1}{L_{i,j}}{\Delta\B{i,j}}=0$. If this was not the case, there would exist a module $M$ with socle $L_{i,l_i}$, having a unique composition factor of type $L_{x,l_x}$ and satisfying $[M:L_{i,j}]=2$. According to parts 6 and 7 of Theorem \ref{thm:newprop}, $M$ would have to be a standard module. This cannot happen as $[M:L_{i,j}]=2$. This shows that the $K$-modules $\End{B}{\Delta\B{i,j}}$ and $\Hom{B}{\Delta\B{i,j+1}}{\Delta\B{i,j}}$ are isomorphic.

For the claim about $R_A$, recall that $\Delta\B{i,1}=\Hom{A}{G}{L_i}$. Since the functor $\Hom{A}{G}{-}$ is fully faithful, $\End{R_A}{\Delta\B{i,1}}\cong\End{A}{L_{i}}$.
\end{proof}

By Proposition \ref{prop:tb1}, the $R_A$-module $T\B{k,l}$ is contained in $\delta_{L-l+1}\B{Q_{k,L}}$. We will show that the maps in $\Hom{R_A}{\delta_{L-l+1}\B{Q_{k,L}}}{\delta_{L-j+1}\B{Q_{i,L}}}$ give rise to maps in $\Hom{R_A}{T\B{k,l}}{T\B{i,j}}$ via restriction. This is the final piece needed to prove Theorem \ref{customthm:isoBlast}. 
\begin{prop}
\label{prop:karinclaim2}
Suppose that $A$ satisfies $\LL{P_i}=\LL{Q_i}=L$ for all $i$, $1 \leq i \leq n$, and assume that all projectives $P_i$ and all injectives $Q_i$ are rigid. Consider a morphism
\[
f_* : \delta_{L-l+1}\B{Q_{k,L}} \longrightarrow \delta_{L-j+1}\B{Q_{i,L}}.
\]
Then $f_*(T\B{k,l})\subseteq T\B{i,j}$.
\end{prop}
\begin{proof}
Because $\Hom{A}{G}{-}$ is a full functor, then $f_*=\Hom{A}{G}{f}$ for a map $f:\SOC{L-l+1}{Q_k} \longrightarrow \SOC{L-j+1}{Q_i}$ in $\Mod{A}$. Note that
\[\Ker{f} \supseteq \RAD{L-j+1}{\B{\SOC{L-l+1}{Q_k}}} = \SOC{\max\{j-l, 0\}}{Q_k}\]
as $Q_k$ is a rigid module. Write $z:=\max\{j-l, 0\}$. Since $T(k,l) \subseteq \delta_{L-l+1}\B{Q_{k,L}}$ and $z \leq L-l+1$, then 
\[
\delta_{z}\B{T(k,l)} \subseteq \delta_{z}\B{\delta_{L-l+1}\B{Q_{k,L}}} = \delta_{z}\B{Q_{k,L}}.
\]
Observe that,
\[
\delta_{z}\B{Q_{k,L}} = \Hom{A}{G}{\SOC{z}{Q_k}} \subseteq \Hom{A}{G}{\Ker{f}}=\Ker{f_*},
\]
so $\delta_{z}(T(k,l))$ is contained in the kernel of $f_*|_{T(k,l)}$. In other words, the module $f_*(T(k,l))$ is isomorphic to a quotient of $T(k,l)/ \delta_{z}(T(k,l))$. Theorem \ref{thm:claim2} implies that all composition factors of $T(k,l)/ \delta_{z}(T(k,l))$ are of the form $L_{x,y}$, with $y \geq l+z\geq j$. Therefore all composition factors of $f_*\B{T(k,l)}$ are of the form $L_{x,y}$, with $(x,y) \not \vartriangleright (i,j)$. By Lemma $5.7$ in \cite{MR3510398}, the module $f_*\B{T(k,l)}$ must be contained in $T(i,j)$.
\end{proof}

\begin{proof}[Proof of Theorem \ref{customthm:isoBlast}]
Consider the morphism of Artin $K$-algebras
\[\varphi : \End{R_A}{\bigoplus_{(i,j) \in \Lambda} \delta_{j}\B{Q_{i,L}}} \longrightarrow \End{R_A}{\bigoplus_{(i,j) \in \Lambda} T(i,j)}=\mathcal{R}\B{R_A}\op ,\]
which sends each map $g\in\END{R_A}{\bigoplus_{i=1}^n \bigoplus_{j=1}^L \delta_{L-j+1}(Q_{i,L})}$ to the corresponding restriction to $\bigoplus_{i=1}^n \bigoplus_{j=1}^L T(i,j)$. According to Proposition \ref{prop:karinclaim2}, $\varphi$ is well defined. Moreover, if $g\neq 0$ then $\varphi (g) \neq 0$, as the modules $\delta_{L-j+1}(Q_{i,L})$ have simple socle. So $\varphi$ is an injective morphism of $K$-algebras, and in particular, a monomorphism of modules in $\Mod{K}$. Proposition \ref{prop:karinclaim1} implies that $\varphi$ is a bijection.

As $\delta_{j}(Q_{i,L})=\Tom{A}{G}{\SOC{j}{Q_i}}$, then
\[
\End{R_A}{\bigoplus_{(i,j) \in \Lambda} \delta_{j}\B{Q_{i,L}}} \cong \End{A}{\bigoplus_{i=1}^n \bigoplus_{j=1}^L \SOC{j}{Q_i}} =(S_A)\op,
\]
using that $\Hom{A}{G}{-}$ is a fully faithful functor. Thus the algebras $\mathcal{R}(R_A)$ and $S_A$ are isomorphic. The identity $S_A \cong (R_{A\op})\op$ was established in Subsection \ref{subsec:motivation}.
\end{proof}
\begin{rem}
Note that the class of connected selfinjective algebras with radical cube zero but radical square nonzero satisfies the conditions of Theorem \ref{customthm:isoBlast}. This class contains several important examples and was studied in \cite{ES}.
\end{rem}

\begin{rem}
Let $A$ be an Artin algebra satisfying $A \cong A \op$. Suppose further that $A$ has rigid projectives and injectives, and assume they all have the same Loewy length. Then, by Theorem \ref{customthm:isoBlast}
\begin{equation}
\label{eq:thmbop}
\mathcal{R}\B{R_A}\cong S_A \cong \B{R_{A\op}}\op \cong \B{R_{A}}\op.
\end{equation}

In particular, the identity \eqref{eq:thmbop} holds when $A$ is a block of weight $1$ of a symmetric group algebra. According to \cite{Scopes2}, blocks of symmetric group algebras of weight $2$ must also satisfy the identity \eqref{eq:thmbop} when the base field has characteristic $p>2$.

Take any group algebra $KG$ where $G$ is a finite $p$-group and $K$ is a field of characteristic $p$. In this setting, $KG$ is a local symmetric algebra, and the only projective indecomposable $KG$-module is rigid: this follows from Jennings' Theorem (see \cite{MR0004626}, and also Theorem $3.14.6$ and Corollary $3.14.7$ in \cite{MR1644252}). Thus, the identity \eqref{eq:thmbop} holds for $A=KG$.

Finally, note that \eqref{eq:thmbop} also holds when $A$ is a preprojective algebra of type A$_n$.
\end{rem}

\section{Cartan matrices and multiplicities}
\label{sec:furtherremarks}
In this section we describe and compare the Cartan matrices of the algebras $R_A$, $\mathcal{R}\B{R_A}$ and $S_A$. Our ultimate goal is to demonstrate that the assumptions in the statement of Theorem \ref{customthm:isoBlast} are, in a certain sense, the minimal requirements for this result to hold. 

To make our arguments simpler, we will work, throughout this section, with \emph{finite-dimensional $K$-algebras $A$ satisfying $\dim \END{A}{L}=1$ for every simple module $L$}. We start by setting some notation. 

If $M$ is a module, we write $[M]$ for its image in the \emph{Grothendieck group} $G_0\B{A}$. Recall that a complete list of pairwise nonisomorphic simple $A$-modules $L_i$, with $i=1, \ldots,n$, gives rise to the $\mathbb{Z}$-basis $\{[L_i]:\,i=1, \ldots,n \}$ of $G_0\B{A}$.

The \emph{Cartan matrix} of $A$ will be denoted by $C(A)$. This is the $n\times n$ matrix whose column with label $j$ has the composition factors of the projective $P_j$. Here the integer $n$ corresponds to the number of isomorphism classes of simple $A$-modules. The entry $ij$ of $C(A)$ is given by $[P_j:L_i]$. 

Due to our assumptions about the simple $A$-modules, we have that $[P_j:L_i]=[Q_i:L_j]$, so the composition factors of the injective incomposable $A$-module $Q_i$ are recorded in the $i^{\text{th}}$ row of $C(A)$. 

Using Lemma \ref{lem:seraultimo}, our assumptions about $A$, and basic properties of quasihereditary algebras, it is not difficult to conclude that the ADR algebra $R_A$ of $A$ still satisfies $\dim\END{R_A}{L_{i,j}}=1$ for every $(i,j)$ in $\Lambda$. By similar arguments, the corresponding algebras $\mathcal{R}\B{R_A}$ and $S_A$ also satisfy the respective condition on simple modules.



\subsection{The Cartan matrix of \texorpdfstring{$R_A$}{[R]}}
The column of $C(R_A)$ associated with the label $(k,l)$ encodes the composition factors of the projective $P_{k,l}$, whereas the row with label $(i,j)$ describes the composition factors of $Q_{i,j}$.

\subsubsection{Projective \texorpdfstring{$R_A$}{[R]}-modules}
Each $P_{k,l}$ is filtered by standard modules. According to Theorem \ref{thm:socdelta}, the multiplicity of $\Delta\B{i,j}$ in $P_{k,l}$ corresponds to the multiplicity of the simple $A$-module $L_i$ in the $j^{\text{th}}$ socle layer of $P_k/\RAD{l}{P_k}$. That is,
\[(P_{k,l}:\Delta\B{i,j}) = [\USOC{j}{(P_k/\RAD{l}{P_k})}: L_i].
\]
We know the composition factors of the modules $\Delta\B{i,j}$ (these are $L_{i,j}, \ldots L_{i,l_i}$). Hence the Cartan matrix $C(R_A)$ is completely determined by the socle series of the radical quotients of the projective $A$-modules. 

The converse of the previous statement is also true: the socle quotients of $P_k/\RAD{l}{P_k}$ can be read off from the Cartan matrix of $R_A$. To see this, note that $[P_{k,l}:L_{i,j}]$ counts the number of factors $\Delta\B{i,y}$, $1 \leq y \leq j$, in a $\Delta$-filtration of $P_{k,l}$. Therefore,
\begin{equation}
\label{eq:karin3}[P_{k,l}:L_{i,j}] - [P_{k,l}:L_{i,j-1}]=(P_{k,l}:\Delta\B{i,j}) =[\USOC{j}{(P_k/\RAD{l}{P_k})}: L_i],
\end{equation}
for $j>1$. We also deduce the following identity 
\begin{equation}
\label{eq:karin4}
[P_{k,l}:L_{i,j}]=\sum_{y=1}^j(P_{k,l}:\Delta\B{i,y})=[\SOC{j}{(P_k/\RAD{l}{P_k})}: L_i].
\end{equation}
\subsubsection{Injective \texorpdfstring{$R_A$}{[R]}-modules}
By parts 4 and 5 of Theorem \ref{thm:newprop}, $Q_{i,j}$ has a $\nabla$-filtration with quotients $\nabla (i, y)$ for $1\leq y\leq j$, each of these occurring exactly once. Therefore, the composition factors of the costandard $R_A$-modules can be totally described in terms of the rows of $C(R_A)$. 
Namely,
\begin{align*}
[Q_{i,1}]  =  & [\nabla(i,1)] \\
[Q_{i,2}] =    &  [\nabla(i,1)] + [\nabla(i,2)]\\
\vdots & \cr
[Q_{i, l_i}]  =   &  [\nabla(i,1)] + [\nabla(i,2)] + \cdots + [\nabla(i, l_i)].
\end{align*}

\begin{lem}
\label{lem:karin2}
We have that
\[ [\nabla\B{i,j}]=[Q_{i,j}] - [Q_{i, j-1}]\]
for $1< j\leq l_i$ and $[Q_{i,1}] = [\nabla(i, 1)]$, that is, for $j>1$, the composition factors of $\nabla\B{i,j}$ can be computed by subtracting the row $(i,j-1)$ from the row $(i,j)$ of $C(R_A)$.
\end{lem}

We can also describe the composition factors of the tilting $R_A$-modules $T\B{i,j}$ using the Cartan matrix of $R_A$. By part 5 of Theorem \ref{thm:newprop},
\[[T(i,j)] = [Q_{i, l_i}] - [Q_{i, j-1}].\]
That is, one can compute the composition factors of $T(i,j)$ by taking the difference of two rows in $C(R_A)$. 

\subsection{The Cartan matrix of \texorpdfstring{$\mathcal{R}\B{R_A}$}{[R(R)]}}
\label{subsec:cartan}
The Cartan matrix $C(\mathcal{R}\B{R_A})$ of $\mathcal{R}\B{R_A}$ has entries
\begin{equation}
\label{eq:ringeldualadr}
[P_{k,l}':L_{i,j}']=\dim\Hom{\mathcal{R}\B{R_A}}{P_{i,j}'}{P_{k,l}'}=\dim\Hom{R_A}{T\B{i,j}}{T\B{k,l}}.
\end{equation}
Note that
\[ 
\dim\Hom{R_A}{T\B{i,j}}{T\B{k,1}}=\dim\Hom{R_A}{T\B{i,j}}{Q_{k,l_k}}=[T\B{i,j}:L_{k,l_k}].
 \]
Using that $\Ext{R_A}{1}{\mathcal{F}\B{\Delta}}{\mathcal{F}\B{\nabla}}$ vanishes, together with Theorem \ref{thm:newprop}, we deduce that
\begin{multline*}
\dim\Hom{R_A}{T\B{i,j}}{T\B{k,l}}\\
\begin{aligned}
&=\dim\Hom{R_A}{T\B{i,j}}{Q_{k,l_k}} -\dim\Hom{R_A}{T\B{i,j}}{Q_{k,l-1}}\\
&=[T(i,j):L_{k,l_k}]- [T(i,j):L_{k,l-1}]
\end{aligned}
\end{multline*}
for $l>1$. That is, the entries of $C(\mathcal{R}(R_A))$ are given by
\[
[P_{k,l}':L_{i,j}']=[T\B{i,j}:L_{k,l_k}]-[T\B{i,j}:L_{k,l-1}]
\]
for $l>1$. The identity $Q_{i,j-1} \cong Q_{i,l_i}/T\B{i,j}$ implies the following result.
\begin{cor}
\label{cor:cartan}
We have
\begin{equation}
\label{eq:cartan}
[P_{k,l}':L_{i,j}']=[Q_{i,l_i}:L_{k,l_k}] - [Q_{i,j-1}:L_{k,l_k}]-[Q_{i,l_i}:L_{k,l-1}] + [Q_{i,j-1}:L_{k,l-1}],
\end{equation}
where $Q_{i,0}:=0$ and $[N:L_{k,0}]:=0$ for $N$ in $\Mod{\mathcal{R}(R_A)}$. In particular, the Cartan matrix of $\mathcal{R}\B{R_A}$ is determined by the Cartan matrix of $R_A$. 
\end{cor}

\begin{rem}
\label{rem:stuttgart}
The result above can be stated more generally for RUSQ (and dually for LUSQ) algebras. In fact, if $B$ is a RUSQ algebra satisfying $\dim \End{B}{L}=1$ for every simple module $L$, then the Cartan matrix of $\mathcal{R}(B)$ is determined by the Cartan matrix of $B$ via the formula in \eqref{eq:cartan}. To deduce that the Cartan matrix of $\mathcal{R}(B)$ is determined by the Cartan matrix of $B$ for $B$ a LUSQ algebra, note that: $B\op$ is RUSQ, $C(B\op)=C(B)^{\operatorname{T}}$ and $\mathcal{R}\B{B\op}\cong \mathcal{R}\B{B}\op$.
\end{rem}

%
%
%
%

\subsection{The Cartan matrix of \texorpdfstring{$S_A$}{[S]}}
\label{subsec:cartan1}
The Cartan matrix of $S_A$ has entries
\[
[P_{k,l}^{S_A}:L_{i,j}^{S_A}]=\dim\Tom{S_A}{P_{i,j}^{S_A}}{P_{k,l}^{S_A}}=\dim\Hom{A}{\SOC{j}{Q_i}}{\SOC{l}{Q_k}}.
\]
Observe that
\[
\Tom{A}{\SOC{j}{Q_{i}}}{\SOC{l}{Q_{k}}}\cong \Tom{A}{\SOC{j}{Q_{i}}/\RAD{l}{\B{\SOC{j}{Q_{i}}}}}{\SOC{l}{Q_{k}}},
\]
and $\SOC{l}{Q_{k}}$ is an injective indecomposable $A/\RAD{k}{A}$-module. Therefore, we have the following formula:
\begin{equation}
\label{eq:saadr}
[P_{k,l}^{S_A}:L_{i,j}^{S_A}]=[\SOC{j}{Q_{i}}/\RAD{l}{\B{\SOC{j}{Q_{i}}}}:L_k].
\end{equation}

\subsection{Comparing \texorpdfstring{$C(\mathcal{R}(R_A))$}{[C(R(R))]} with \texorpdfstring{$C(S_A)$}{[C(S)]}}
In Subsection \ref{subsec:motivation}, we have looked at some facts hinting at a relationship between the quasihereditary algebras
\[(\mathcal{R}(R_A), \Lambda, \unlhd\op) \quad\text{and}\quad(S_A, \Lambda_{A\op}, \unlhd_{A\op}).\]

In particular, we have seen that it would be reasonable to require that $\LL{P_i}=\LL{Q_i}=l_i$ for all $i$, in order to have an isomorphism between $\mathcal{R}(R_A)$ and $S_A$. This requirement would at least assure that $|\Lambda|=|\Lambda_{A\op}|$. We have also seen that it would be natural to map an idempotent $\xi_{(i,j)}'$ in $\mathcal{R}(R_A)$ associated with the label $(i,j)$ to an idempotent $\varepsilon_{[i,l_i-j+1]}$ in $S_A$ associated with the label $[i,l_i-j+1]$.

For this correspondence to yield an isomorphism, the numbers
\begin{gather*}
\dim \xi_{(i,j)}'\mathcal{R}(R_A)\xi_{(k,l)}'=[P_{k,l}':L_{i,j}']\\
\dim \varepsilon_{[i,l_i-j+1]}S_A\varepsilon_{[k,l_k-l+1]}=[P_{k,l_k-l+1}^{S_A}:L_{i,l_i-j+1}^{S_A}]
\end{gather*}
should coincide. That is, the entry $(i,j)(k,l)$ of $C(\mathcal{R}(R_A))$ must match with the entry $[i,l_i-j+1][k,l_k-l+1]$ of $C(S_A)$ for every $i$,$k$,$j$ and $l$.

Our aim is to show that the assumptions in the statement of Theorem \ref{customthm:isoBlast} are somehow minimal. For this, consider the conditions:
\begin{description}
\item[(B1)\label{item:B1n}] $\LL{P_i}=\LL{Q_i}=l_i$ for all $1 \leq i \leq  n$;
\item[(B2)\label{item:B2n}] $[P_{k,l}':L_{i,j}']=[P_{k,l_k-l+1}^{S_A}:L_{i,l_i-j+1}^{S_A}]$ for all $1 \leq i,k \leq n$, $1\leq j \leq l_i$ and $1 \leq l \leq  l_k$.
\end{description}
We wish to prove the following result.
\begin{customthm}{B}
\label{thm:karin}
Let $A$ be a finite-dimensional connected $K$-algebra. Suppose that $\dim \End{A}{L_i} =1$ for every $i$, and assume that \ref{item:B1n} and \ref{item:B2n} hold. Then:
\begin{enumerate}
\item all the Loewy lengths $l_i$ are the same (i.e.~$l_i=l_k$ for all $i$ and $k$);
\item each projective $P_i$ is rigid;
\item each injective $Q_i$ is rigid.
\end{enumerate}
\end{customthm}
We prove this theorem in a number of steps. 

\begin{lem}
\label{lem:karinlast}
Assume \ref{item:B1n} and \ref{item:B2n}. Then for all $i$ we have
\[
T\B{i,l_i}=\Delta\B{i,l_i}=\nabla\B{i,l_i}= L_{i,l_i}.
\]
\end{lem}
\begin{proof}
According to part 5 of Theorem \ref{thm:newprop}, we have $T\B{i,l_i}=\nabla\B{i,l_i}$. We want to show that $\nabla\B{i,l_i}$ is a simple module. Note that $[\nabla(i,l_i):L_{i,j}]=0$ for $j\neq l_i$, as $(i,l_i)\lhd (i,j)$ for $j \neq l_i$. If the module $T\B{i,l_i}=\nabla\B{i,l_i}$ is not simple, then it has some factor $\Delta\B{k,l}$ for $k\neq i$. That is, $\Tom{A}{T(i,l_i)}{T(k,l)}\neq 0$ for some $k\neq i$. By \eqref{eq:ringeldualadr}, \ref{item:B2n} and \eqref{eq:saadr}, we have then that $L_k$ occurs in $\SOC{1}{Q_i}$, which is a contradiction because $i\neq k$.
\end{proof}

The next proposition will be crucial in the proof of part 1 of Theorem \ref{thm:karin}.
\begin{prop}
\label{prop:imp}
Assume that \ref{item:B1n} and \ref{item:B2n} hold for $A$ and suppose that $i\neq k$. If $\Ext{A}{1}{L_i}{L_k} \neq 0$ then $l_k \leq l_i$.
\end{prop}
\begin{proof}
By assumption, we must have
\[
[\SOC{2}{Q_k}/\Rad{(\SOC{2}{Q_k})}:L_i]\neq 0.
\]
Using \eqref{eq:saadr}, \ref{item:B2n} and \eqref{eq:ringeldualadr}, we deduce that
\[
\Hom{R_A}{T(k,l_k-1)}{T(i,l_i)}\neq 0.
\]
From Lemma \ref{lem:karinlast}, it follows that $[T(k,l_k-1):L_{i,l_i}]\neq 0$. Parts 4 and 5 of Theorem \ref{thm:newprop} imply that either $[T(k,l_k):L_{i,l_i}]\neq 0$, or $[T(k,l_k-1)/T(k,l_k):L_{i,l_i}]\neq 0$, where the quotient $T(k,l_k-1)/T(k,l_k)$ is isomorphic to $\nabla (k,l_k-1)$. Using that $i\neq k$, together with Lemma \ref{lem:karinlast}, we conclude that $[T(k,l_k):L_{i,l_i}]=0$. Thus $[\nabla(k,l_k-1):L_{i,l_i}]\neq 0$. As a consequence, $(P_{i,l_i}:\Delta (k,l_k-1))\neq 0$, using Brauer--Humphreys reciprocity for quasihereditary algebras (see \cite[Lemma $2.5$]{MR1211481}). The identity \eqref{eq:karin3} implies that $[\USOC{l_k-1}{P_i}:L_k]\neq 0$. So $P_i$ has Loewy length at least $l_k-1$. In fact, as $i\neq k$, one deduces that $l_i>l_k-1$, or equivalently $l_k\leq l_i$.
\end{proof}

A key argument in proof of Theorem \ref{thm:karin} is the fact that the conditions \ref{item:B1n} and \ref{item:B2n} are `symmetric'. We give an informal explanation for this phenomenon. The axioms \ref{item:B1n} and \ref{item:B2n} are saying that $C(\mathcal{R}(R_A))$ and $C(S_A)$ coincide (up to a suitable simultaneous permutation of rows and columns). As explained in Remark \ref{rem:stuttgart}, the matrix $C(\mathcal{R}(B))$ is determined by $C(B)$ when $B$ is a LUSQ algebra. So the Cartan matrices of the algebras $\mathcal{R}(\mathcal{R}(R_A))\cong R_A$ and $\mathcal{R}(S_A)$ should still coincide when \ref{item:B1n} and \ref{item:B2n} hold for $A$. Recall the discussion in Subsection \ref{subsec:motivation}. Note that
\[R_A\cong (S_{A\op})\op\quad\text{ and }\quad\mathcal{R}(S_A)\cong\mathcal{R}((R_{A\op})\op)\cong \mathcal{R}(R_{A\op})\op,\]
therefore $C(S_{A\op})=C(R_A)^{\operatorname{T}}$ and $C(\mathcal{R}(R_{A\op}))=C(\mathcal{R}(S_A))^{\operatorname{T}}$. It seems then natural that \ref{item:B1n} and \ref{item:B2n} hold for the underlying algebra $A$ if and only if \ref{item:B1n} and \ref{item:B2n} hold for $A\op$.
\begin{lem}
\label{lem:imp}
Let $A$ be a finite-dimensional $K$-algebra, and let $\dim \END{A}{L_i} =1$ for every $i$. Assume that \ref{item:B1n} and \ref{item:B2n} hold for $A$. Then $A\op$ is a finite-dimensional $K$-algebra satisfying $\dim \END{A\op}{L_i^{A\op}} =1$ for every $i$. Moreover, the conditions \ref{item:B1n} and \ref{item:B2n} hold for $A\op$.
\end{lem}
\begin{proof}
The first part of the statement of the lemma is evident. It is also clear that $A\op$ satisfies \ref{item:B1n}. We show that \ref{item:B2n} holds for $A\op$. Using Corollary \ref{cor:cartan}, duality, and condition \ref{item:B2n} for $A$, we get
\begin{multline*}
[P_{k,l}^{\mathcal{R}(R_{A\op})} :L_{i,j}^{\mathcal{R}(R_{A\op})} ]  \\
\begin{aligned}
&= [Q_{i,l_i}^{R_{A\op}}:L_{k,l_k}^{R_{A\op}}] - [Q_{i,j-1}^{R_{A\op}}:L_{k,l_k}^{R_{A\op}}]-[Q_{i,l_i}^{R_{A\op}}:L_{k,l-1}^{R_{A\op}}] + [Q_{i,j-1}^{R_{A\op}}:L_{k,l-1}^{R_{A\op}}] \\
& = [P_{i,l_i}^{S_{A}}:L_{k,l_k}^{S_{A}}] - [P_{i,j-1}^{S_A}:L_{k,l_k}^{S_{A}}]-[P_{i,l_i}^{S_{A}}:L_{k,l-1}^{S_{A}}] + [P_{i,j-1}^{S_{A}}:L_{k,l-1}^{S_{A}}] \\
& = [P_{i,1}':L_{k,1}'] - [P_{i,l_i-j+2}':L_{k,1}']-[P_{i,1}':L_{k,l_k-l+2}'] + [P_{i,l_i-j+2}':L_{k,l_k-l+2}'].
\end{aligned}
\end{multline*}
By applying Corollary \ref{cor:cartan} to the last expression, it follows that
\begin{align*}
[P_{k,l}^{\mathcal{R}(R_{A\op})} :L_{i,j}^{\mathcal{R}(R_{A\op})} ]  =& [Q_{k,l_k}:L_{i,l_i}] - ([Q_{k,l_k}:L_{i,l_i}] - [Q_{k,l_k}:L_{i,l_i-j+1}]) \\
& -([Q_{k,l_k}:L_{i,l_i}] - [Q_{k,l_k-l+1}:L_{i,l_i}]) \\
&+ [Q_{k,l_k}:L_{i,l_i}] - [Q_{k,l_k-l+1}:L_{i,l_i}] \\
&-[Q_{k,l_k}:L_{i,l_i-j+1}] +[Q_{k,l_k-l+1}:L_{i,l_i-j+1}] \\
=& [Q_{k,l_k-l+1}:L_{i,l_i-j+1}] =[P_{k,l_k-l+1}^{S_{A\op}}:L_{i,l_i-j+1}^{S_{A\op}}].
\end{align*}
This concludes the proof of the proposition.
\end{proof}
\begin{cor}
\label{cor:lasthopefully}
Assume that \ref{item:B1n} and \ref{item:B2n} hold for $A$ and suppose that $i\neq k$. If $\Ext{A}{1}{L_i}{L_k} \neq 0$ then $l_k =l_i$.
\end{cor}
\begin{proof}
The inequality $l_k \leq l_i$ follows by applying Proposition \ref{prop:imp} to $A$. According to Lemma \ref{lem:imp}, \ref{item:B1n} and \ref{item:B2n} also hold for $A\op$. Moreover, note that $\EXT{A\op}{1}{L_k^{A\op}}{L_i^{A\op}} \neq 0$. The inequality $l_i\leq l_k$ then follows by applying Proposition \ref{prop:imp} to $A\op$.
\end{proof}
We finally prove Theorem \ref{thm:karin}.

\begin{proof}[Proof of Theorem \ref{thm:karin}]
We start by showing that all $l_i$ must be equal. For every distinct $i$ and $k$ in $\{1, \ldots, n \}$ there exists a sequence $(i_1, \ldots, i_m)$ with $i_1=i$, $i_m=k$ and $1\leq i_{x} \leq n$, satisfying the following property: for each $1\leq x < m$, either $\EXT{A}{1}{L_{i_x}}{L_{i_{x+1}}}\neq 0$ or $\EXT{A}{1}{L_{i_{x+1}}}{L_{i_{x}}}\neq 0$. Part 1 is then an easy consequence of Corollary \ref{cor:lasthopefully}.

Let now $L$ be the common Loewy length of all projectives and injectives. Note that
\begin{align*}
[P_{k,1}':L_{i,j}']&=[Q_{i,L}:L_{k,L}]-[Q_{i,j-1}:L_{k,L}]\\
&=[P_{k,L}:L_{i,L}]-[P_{k,L}:L_{i,j-1}]\\
&=[\SOC{L}{(P_k/\RAD{L}{P_k})}/\SOC{j-1}{(P_k/\RAD{L}{P_k})}:L_i]\\
&=[P_k/ \SOC{j-1}{P_k}:L_i].
\end{align*}
Here, we have used Corollary \ref{cor:cartan}, duality, and the identity in \eqref{eq:karin4}. Similarly,
\begin{align*}
[P_{k,L}^{S_A}:L_{i,L-j+1}^{S_A}]&=[\SOC{L-j+1}{Q_i}:L_k]\\
&=[P_k/ \RAD{L-j+1}{P_k}:L_i],
\end{align*}
where the first equality follows from \eqref{eq:saadr}, and the second equality follows by duality for $A/\RAD{L-j+1}{A}$. By \ref{item:B2n}, the multiplicities $[P_k/ \SOC{j-1}{P_k}:L_i]$ and $[P_k/ \RAD{L-j+1}{P_k}:L_i]$ coincide for every $i$, $k$ and $j$. As a consequence, the modules $\SOC{j-1}{P_k}$ and $\RAD{L-j+1}{P_k}$ must have the same Jordan-H\"{o}lder length. Since $\RAD{L-j+1}{P_k}$ is contained in $\SOC{j-1}{P_k}$, then these modules are actually equal. Therefore $P_k$ is a rigid module for every $k$.

Observe that $A\op$ is a finite-dimensional connected $K$-algebra. Moreover, $A\op$ satisfies $\dim \END{A\op}{L_i^{A\op}} =1$ for every $i$. According to Corollary \ref{cor:lasthopefully}, $A\op$ also satisfies axioms \ref{item:B1n} and \ref{item:B2n}. Thus, by the previous, the projective $A\op$-modules $P_k^{A\op}=D(Q_k)$ must be rigid. As a consequence, every injective indecomposable $A$-module is rigid.
\end{proof}

\section{Ringel selfdual ADR algebras}
\label{sec:last}
We wish to characterise Ringel selfdual ADR algebras. First, the notion of Ringel selfduality must be rigorously defined. We say that two quasihereditary algebras $(B,\Phi , \sqsubseteq )$ and $(C,\Psi, \preccurlyeq)$ are \emph{equivalent} if the respective categories $\mathcal{F}( \Delta )$ and $\mathcal{F}( \Delta^{C})$ are equivalent. A quasihereditary algebra $(B, \Phi , \sqsubseteq )$ is \emph{Ringel selfdual} if the algebras $(B, \Phi , \sqsubseteq )$ and $(\mathcal{R}(B), \Phi , \sqsubseteq \op)$ are equivalent.

It is not unusual for a quasiherederitary algebra $\B{B, \Phi,\sqsubseteq}$ to be Ringel selfdual. This phenomenon is frequently observed in quasihereditary algebras and highest weight categories arising from the theory of semisimple Lie algebras and algebraic groups. Therefore, it is natural to ask which quasihereditary algebras are Ringel selfdual.

As pointed out in \cite[Appendix, A$.2$]{RingelIyama}, if $\B{B,\Phi, \sqsubseteq}$ is both right and left strongly quasihereditary, then $B$ has global dimension at most 2. For this reason, one should not expect that right strongly quasihereditary algebras are often Ringel selfdual. In particular, one should not expect that $R_A$ is Ringel selfdual. We give necessary and sufficient conditions for an ADR algebra to be Ringel selfdual.

\begin{customthm}{C}
The algebra $\B{R_A, \Lambda, \unlhd}$ is Ringel selfdual if and only if $A$ is a selfinjective Nakayama algebra.
\end{customthm}
\begin{proof}
If $R_A$ is Ringel selfdual, then the indecomposable tilting modules over $R_A$ must coincide with the indecomposable tilting modules over $\mathcal{R}(R_A)$.

The $R_A$-modules $P_{i,1}$, $i=1. \ldots, n$, form a complete list of projective indecomposable modules isomorphic to a standard module (see \cite{MR3510398}, Propositions $3.1$ and $3.4$).

Consider now the Ringel dual $\mathcal{R}(R_A)$ of $R_A$. By Theorem \ref{thm:lastonebynow}, $P_{i,1}'\cong T'\B{i,l_i}$ has a unique $\Delta ' $-filtration, given by
\[
0 \subset P'_{i,l_i} \subset \cdots \subset P'_{i,j} \subset \cdots \subset P'_{i,1}=T'\B{i,l_i},
\]
$P_{i,j}'/P_{i,j+1}'\cong \Delta'\B{i,j}$ and  $T'\B{i,j}\cong P'_{i,1}/ P'_{i,j+1}$. The modules $P'_{i,l_i}$, $i=1. \ldots, n$, form a complete list of projective modules isomorphic to a standard module. By the previous observation about the algebra $R_A$, there must be a bijective correspondence between the labels $(i,1)$ in $\Mod{R_A}$ and the labels $(k,l_k)$ in $\Mod{\mathcal{R}(R_A)}$. 

Consequently, each tilting $R_A$-module $T(i,1)=Q_{i,l_i}$ must correspond bijectively to some tilting $\mathcal{R}(R_A)$-module of the form $T'\B{k,l_{k}}=P'_{k,1}$. By the involutive properties of the Ringel duality (see the proof of Theorem 6 and Lemma 7 in \cite{last}), each projective $\mathcal{R}(R_A)$-module $\Hom{R_A}{T}{T(k,1)}\cong P'_{k,1}$ coincides then with some projective $R_A$-module $\Hom{\mathcal{R}(R_A)}{T'}{T'\B{x,l_{x}}}\cong P_{x,l_{x}}$. Therefore, each injective $R_A$-module $Q_{i,l_i}$ is isomorphic to some projective $R_A$-module of type $P_{x,l_x}$. By definition, we have $P_{x,l_x}=\Hom{A}{G}{P_x}$, and by Theorem \ref{prop:standard}, $Q_{i,l_i}$ is isomorphic to $\Hom{A}{G}{Q_i}$. Since the functor $\Hom{A}{G}{-}$ is fully faithful, it follows that $Q_i\cong P_x$. Thus, $A$ must be a selfinjective algebra. 

According to previous observations we also know that $\Gl{R_A} \leq 2$. Therefore, $\Rad{A}$ lies in $\Add{G}$ by Proposition $2$ in \cite{smalo1978}. Since $A$ is selfinjective, the property $\Rad{A} \in \Add{G}$ implies that the projective(-injective) indecomposable $A$-modules are uniserial. So $A$ is a selfinjective Nakayama algebra.

The converse is a well-known result, and a proof can be found in \cite{T}. Alternatively, note that every connected selfinjective Nakayama algebra $A$ satisfies the assumptions in the statement of Theorem \ref{customthm:isoBlast}, hence we have a structure-preserving isomorphism between $\mathcal{R}(R_A)$ and $S_A$. Now observe that $S_A\cong R_A$ as $A$ is a selfinjective Nakayama algebra. Thus, the ADR algebra of a connected selfinjective Nakayama algebra is Ringel selfdual. Note that ADR algebras and Ringel duals are well behaved with respect to the direct product of algebras, that is $R_{A_1\times A_2}\cong R_{A_1} \times R_{A_2}$ and $\mathcal{R}(B_1 \times B_2)\cong \mathcal{R}(B_1)\times \mathcal{R}(B_2)$. Using that an arbitrary selfinjective Nakayama algebra is the direct product of connected selfinjective Nakayama algebras, we deduce that the ADR algebra of a selfinjective Nakayama algebra is Ringel selfdual.
\end{proof}
\begin{rem}
Note that if $A$ is a Nakayama algebra, then the ADR algebra of $A$ coincides with its Auslander algebra.
\end{rem}

\bibliographystyle{amsplain}
\bibliography{QHAlg}

\end{document}